\documentclass[ 11pt, reqno]{amsart}

\usepackage{amsmath,amssymb,amsxtra,epsf,amscd,graphics,color,array,ulem,etex, colortbl}
\usepackage{mathrsfs}

\usepackage[latin1]{inputenc}
\usepackage[T1]{fontenc}
\usepackage[frenchb, english]{babel}

\usepackage{amsfonts}
\usepackage{amsthm}

\usepackage{mleftright}
\usepackage{ytableau}

\usepackage{enumerate,array,calc}
\usepackage{epsf}
\usepackage{epsfig}

\usepackage{amsfonts,amssymb,wasysym}

\usepackage{hhline}
\usepackage{tikz}
\usetikzlibrary{patterns}
\usepackage{hyperref}
\usepackage{caption}

\newcommand{\g}{\mathfrak{g}}
\newcommand{\ad}{\operatorname{ad}}

\def\ep{\varepsilon}

\def\p{\mathfrak p}

\def\h{\mathfrak h}

\def\a{\mathfrak a}
\def\m{\mathfrak m}
\def\b{\mathfrak b}
\def\n{\mathfrak n}

\def\q{\mathfrak q}
\def\ep{\varepsilon}

\theoremstyle{definition} 
\newtheorem{theorem}{Theorem}[subsection]
\newtheorem{thm}[theorem]{Theorem}

\newtheorem{prop}[theorem]{Proposition}
\newtheorem{lm}[theorem]{Lemma}

\newtheorem{defi}[theorem]{Definition}

\newtheorem{Rq}[theorem]{Remark}

\title[Polynomiality for some In\"on\"u-Wigner Contractions in type A.]{Polynomiality for algebras of invariants associated with some In\"on\"u-Wigner Contractions in type A.}

\author{Florence Fauquant-Millet}

\address{
Université Jean Monnet, Centrale Lyon, INSA Lyon, Universit\'e Lyon 1, CNRS, ICJ UMR5208, 42023 Saint-
Etienne, France.}
\email{florence.millet@univ-st-etienne.fr}

\begin{document}

\begin{abstract}

This paper deals with polynomiality of algebras of symmetric invariants (or generated by symmetric semi-invariants) associated with some particular In\"on\"u-Wigner contractions in type A. More precisely we are interested with contractions of some parabolic subalgebras in type A with respect to their decomposition into their Levi factor and their nilpotent radical, the latter becoming an abelian ideal of the contraction. When the parabolic subalgebra $\p$ has its Levi factor  decomposing into three  blocks, which are symmetric with respect to the antidiagonal, we show in this article that the algebra of symmetric invariants associated with the contraction of the canonical truncation of $\p$ has a Weierstrass section. It follows that such an algebra of invariants is a polynomial algebra, for which we can give the number of algebraically independent homogeneous generators, their weight and degree. The method relies on the construction of an adapted pair for such a contraction. 
When the Levi factor of $\p$ has its central block not of the same size as its two extremal blocks, we also show that the above algebra of symmetric invariants is equal to the algebra generated by symmetric semi-invariants associated with the In\"on\"u-Wigner contraction of  $\p$.
\end{abstract}

\maketitle

{\it Mathematics Subject Classification} : 16 W 22, 17 B 22, 17 B 35.

{\it Key words} : In\"on\" u-Wigner contraction,  parabolic subalgebra, symmetric invariants, semi-invariants, Weierstrass section, polynomiality.

\section{The context.}\label{Intro}

Throughout this paper, we assume that the base field $\mathbb K$ is algebraically closed of characteristic zero and we will denote by $\a$ any algebraic finite-dimensional Lie algebra over $\mathbb K$. Then $S(\a)$ is the symmetric algebra of $\a$ and $U(\a)$ is the universal enveloping algebra of $\a$.  

\subsection{}\label{contr}

If $\a=\mathfrak b\oplus\mathfrak c$ with $\mathfrak c$ being an ideal of $\a$ and $\mathfrak b$ a Lie subalgebra of $\a$, we may define the In\"on\"u-Wigner contraction (or one-parameter contraction) $\widetilde\a$ of $\a$ with respect to this decomposition to be the semi-direct product
$$\widetilde\a=\mathfrak b\ltimes\mathfrak c^a$$
where $\mathfrak c\simeq\mathfrak c^a$ as vector spaces and even as $\mathfrak b$-modules but where $\mathfrak c^a$ is an abelian ideal of $\widetilde\a$ (see for instance \cite[Sec. 4]{Y1} and \cite[Remark 2.3]{Fe}).

As $\mathfrak b$-modules, $\a$ and $\widetilde\a$ are isomorphic (the latter may be viewed as a degeneration of the former).  Moreover $\widetilde\a$ is still a Lie algebra with the Lie bracket $[\,,\,]_{\widetilde\a}$ defined as follows.

\begin{align*}
\forall x,\,x'\in\mathfrak b,\,\forall y,\,y'\in\mathfrak c^a,\,[x,\,x']_{\widetilde\a}=[x,\,x'],\;\;[x,\,y]_{\widetilde\a}=[x,\,y],\;\;[y,\,y']_{\widetilde\a}=0\end{align*}
where $[\,,\,]$ is the Lie bracket in $\a$.
The superscript $a$ in $\mathfrak c^a$ means that $\mathfrak c^a$ is abelian in the contraction $\widetilde\a$.

\subsection{}\label{par}
Let $\g$ be a simple Lie algebra over $\mathbb K$, $\h$ be a Cartan subalgebra of $\g$ and $\pi$ be a set of simple roots of $(\g,\,\h)$.
Let $\p$ be a standard parabolic subalgebra  in $(\g,\,\h,\,\pi)$. Denote by $\mathfrak r$  the standard Levi factor of $\p$ and $\m$ the nilpotent radical of $\p$, so that we have $$\p=\mathfrak r\oplus\m.$$
Then $\mathfrak r$ is a reductive Lie subalgebra of $\p$ and $\m$ is an ideal of $\p$.
According to the notation in subsection \ref{contr}, we denote by $\widetilde\p=\mathfrak r\ltimes\m^a$ the In\"on\"u-Wigner contraction (or one-parameter contraction)  with respect to the above decomposition of $\p$. Then $\m^a$ becomes an abelian ideal of $\widetilde\p$.
 
\subsection{}\label{semi}
If $\b$ is a Lie subalgebra of $\a$, one denotes by $S(\a)^{\b}$ the algebra of symmetric invariants in $S(\a)$  under adjoint action of $\b$.
In particular, one sets $Y(\a)=S(\a)^{\a}$.  
An element $s\in S(\a)$ is called a semi-invariant of weight $\lambda\in\a^*$ if, for all $x\in\a$, we have that
$${\rm ad}\,x(s)=\lambda(x)s$$ where $\rm ad$ denotes the adjoint action of $\a$ on $S(\a)$ extending the Lie bracket in $\a$ by derivation. We denote by $S(\a)_{\lambda}$ the vector subspace of $S(\a)$ formed by all semi-invariant elements of weight $\lambda$. Then the algebra $Sy(\a)$ generated by symmetric semi-invariants in $S(\a)$ under adjoint action is given by the equality :
$$Sy(\a)=\bigoplus_{\lambda\in\a^*}S(\a)_{\lambda}.$$
The set of weights $\Lambda(\a)$ of $Sy(\a)$ is given by the equality :
 $$\Lambda(\a)=\{\lambda\in\a^*\mid S(\a)_{\lambda}\neq\{0\}\}.$$
Since the symmetric invariants of $S(\a)$ are semi-invariants of weight zero, we have obviously that $Y(\a)\subset Sy(\a)$. Denote by $\a'$ the derived subalgebra of $\a$. There exists a socalled canonical truncation $\a_{\Lambda}$ of $\a$ such that 
\begin{equation}Sy(\a)=Y(\a_{\Lambda})=Sy(\a_{\Lambda}).\label{semitrunc}\end{equation} The canonical truncation $\a_{\Lambda}$ of $\a$ is the largest ideal of $\a$ which vanishes on the set $\Lambda(\a)$ : it always contains $\a'$. In other words one has that
\begin{equation}\a_{\Lambda}=\cap_{\lambda\in\Lambda(\a)}\ker(\lambda).\label{deftrunc}\end{equation}

\subsection{}\label{reg}

For $\xi\in\a^*$, set $\a_{\xi}=\{x\in\a\mid\forall y\in\a,\; \xi([x,\,y])=0\}$, and denote by $${\rm index}\,\a=\inf_{\xi\in\a^*}\dim\a_{\xi}$$ the index of $\a$, as defined for example in \cite[1.11.6]{D}.
Since $\a$ is algebraic, the index of $\a$ is also equal to the minimal codimension of a coadjoint orbit.
For instance if $\a$ is abelian, then ${\rm index}\,\a=\dim\a.$
An element $\xi\in\a^*$ such that $\dim\a_{\xi}={\rm index}\,\a$ is called a regular element of $\a^*$. The subset of $\a^*$ formed by regular elements is a non-empty Zariski open subset of $\a^*$ by \cite[1.11.5]{D} for example.

\subsection{} \label{GKdim}

For every subalgebra $S\subset S(\a)$ denote by ${\rm GKdim}\, S$ the Gelfand-Kirillov dimension of $S$ which is also equal to the transcendence degree over $\mathbb K$ of the field of fractions of $S$.
As a consequence of a result of Chevalley-Dixmier (\cite[Lem. 7]{D0}) we have that 
\begin{equation}{\rm index}\,\a_{\Lambda}={\rm GKdim}\,Y(\a_{\Lambda})={\rm GKdim}\,Sy(\a).\label{GKdimindex}\end{equation}

\subsection{}\label{Carnot} 
By \cite[Rem. 2.4.1, Cor. 2.6.5]{F01}, if $\a$ is a maximal standard parabolic subalgebra $\p$ in  $(\g,\,\h,\,\pi)$, then $\p_{\Lambda}=\p'$ and $\widetilde\p_{\Lambda}=\widetilde\p'=\mathfrak r'\ltimes\m^a$. More generally for a standard parabolic subalgebra $\p$ in $(\g,\,\h,\,\pi)$, there exists a subspace $\h_{\Lambda}\subset\h$ (which is not the canonical truncation of $\h$, but which is equal to $\p_{\Lambda}\cap\mathfrak z(\mathfrak r)$ where $\mathfrak z(\mathfrak r)$ is the centre of the standard Levi factor $\mathfrak r$ of $\p$)  such that $\p_{\Lambda}=\mathfrak r'\oplus\h_{\Lambda}\oplus\m$ and $\h_{\Lambda}$ need not be reduced to $\{0\}$. For instance $\h_{\Lambda}\neq\{0\}$  if we consider, in type A, a parabolic subalgebra $\p$ whose Levi factor is formed by three  blocks which are symmetric with respect to the antidiagonal. 
When $\h_{\Lambda}\neq\{0\}$ it is not clear in general what should be the canonical truncation $\widetilde\p_{\Lambda}$ of the contraction $\widetilde\p$ and then what should be the algebra $Sy\big(\widetilde\p\big)$ generated by symmetric semi-invariants  in terms of algebra of symmetric invariants (in view of Eq. \ref{semitrunc}).

\subsection{}\label{caseBorel}
In the notation of subsections \ref{par} and \ref{semi}, we continue in this article the study of polynomiality of the algebra $Sy\big(\widetilde\p\big)$ that we have initiated in our papers \cite{F00} and \cite{F01}.
 Since $Sy\big(\widetilde\p\big)=S\big(\widetilde\p\big)^{\widetilde\p'}$, we may observe that, if the algebra $Sy\big(\widetilde\p\big)$ is polynomial, then it is the same for the algebra of invariants $S(\m)^{\mathfrak r'}\subset Sy\big(\widetilde\p\big)$ by \cite[Thm. 2.3]{P}. Suppose for a moment that $\p$ is the standard Borel subalgebra in $(\g,\,\h,\pi)$. Since $S(\m)^{\mathfrak r'}=S(\m)$ in this case,  every root vector in $\m$ is a semi-invariant. Now an element in $\h$ which vanishes on the different weights of root vectors in $\m$ is equal to zero. It follows that the canonical truncation of the contraction of the Borel subalgebra is equal to its derived subalgebra  which is here equal to $\m=\m^a$ as $U\big(\widetilde\p\big)$-modules. Then in this case one has that $Sy\big(\widetilde\p\big)=S(\m^a)=S(\m)$, which  is obviously polynomial. 
 
 Now return to the general case. By \cite[Cor. 7]{Kn}, when the $\mathfrak r'$-module $\m$ has no nonzero $\mathfrak r'$-invariant vectors and if $S(\m)^{\mathfrak r'}$ is polynomial, then we have that
$$\dim\m\le 2\dim\mathfrak r'.$$

It follows that in general, polynomiality for $Sy\big(\widetilde\p\big)$ fails. For example, if $\g$ is simple of type ${\rm A}_6$ and $\p$ is the parabolic subalgebra of $\g$ whose Levi factor consists in two extremal blocks of size two and three central blocks of size one, then we have that $Sy\big(\widetilde\p\big)$ is not polynomial, because $S(\m)^{\mathfrak r'}$ is not. More generally, for a parabolic subalgebra $\p$ not equal to the Borel subalgebra of $\g$, polynomiality of $Sy\big(\widetilde\p\big)$ may fail if the Levi factor of the parabolic subalgebra $\p$ consists in too many blocks.

\subsection{}
 In this article we will focus on $\g$ simple of type A and on the In\"on\"u-Wigner contraction $\widetilde\p$ of a standard parabolic subalgebra $\p$ in $(\g,\,\h,\,\pi)$.
If the standard Levi factor of $\p$ consists in two blocks, namely if $\p$  is maximal, we have in this case $\widetilde\p=\p$ because the nilpotent radical of $\p$  is already abelian  in $\p$. If the standard Levi factor $\mathfrak r$ of $\p$ consists in one block then obviously $\p=\g=\mathfrak r=\widetilde\p$. Moreover we want to give in this article examples of an In\"on\"u-Wigner contraction $\widetilde\p$ of a parabolic subalgebra $\p$ for which the canonical truncation $\widetilde\p_{\Lambda}$ is not necessarily equal to the derived subalgebra $\widetilde\p'$ of the contraction and for which we have a chance to have polynomiality for the algebra $Sy\big(\widetilde\p\big)$. 
That is why we will consider In\"on\"u-Wigner contractions of parabolic subalgebras in $\g$ simple of type A whose Levi factor consists in three blocks for which the two extremal blocks are {\it of the same size}  (in other words, these three blocks are symmetric with respect to the antidiagonal).

\subsection{}
The definition of an adapted pair for $\a$ may be found in \cite[6.1]{JS} for example. This is a pair $(h,\,y)\in\a\times\a^*$ formed by an element $h\in\a$ which is $\ad$-semisimple (as endomorphism of $\a^*$) and an element $y\in\a^*$ which is regular (in the sense of subsection \ref{reg}), such that 
$$\ad h(y)=-y.$$
Here ad denotes  the coadjoint action of $\a$ on its dual space $\a^*$.

Recall that $S(\a)$ may be identified with the algebra of polynomial functions $\mathbb K[\a^*]$ on $\a^*$. A Weierstrass section for $Y(\a)$ (see \cite[2.2.1]{Po}) 
 is an affine subset $y+V$ of $\a^*$ (with $y\in\a^*$ and $V$ a vector subspace of $\a^*$) such that restriction of functions induces an algebra isomorphism between $Y(\a)$ and the algebra of polynomial functions $\mathbb K[y+V]$ on $y+V$. Since $\mathbb K[y+V]$ is isomorphic to $S(V^*)$, it follows that $Y(\a)$ is a polynomial algebra over $\mathbb K$ in $\dim V$ algebraically independent generators if it has a Weierstrass section $y+V$ (but the converse is not true in general). By \cite[Chap. VI, Sec. A, 1]{F} (see also \cite[Lem. 3.2]{FJ4}) a Weierstrass section for $Y(\a)$ is also an affine slice to the coadjoint action of $\a$  in the sense of \cite[7.3, 7.4]{J8} and \cite[1.5]{FJ4}, which means notably that  this slice meets every coadjoint orbit through this slice at exactly one point and transversally. 

\subsection{}
As we already said, if the Levi factor of the parabolic subalgebra $\p$ of $\g$ simple consists in one block, then $\widetilde\p=\g$ and in this case, theorems of Chevalley and of Harish-Chandra give the polynomiality of the algebra of symmetric invariants $Y(\g)=Sy(\g)=Y(\g_{\Lambda})$ and a theorem of Kostant gives an adapted pair for $\g$ and a Weierstrass section for $Y(\g)$ (also called a Kostant slice in this case).
On the other hand, for $\g$ simple of type ${\rm A}$, any maximal parabolic subalgebra $\p$ of $\g$ coincides with its contraction $\widetilde\p$. Moreover adapted pairs were already constructed in type A for the  canonical truncation $\p_{\Lambda}$ of any parabolic (and even of any biparabolic) subalgebra $\p$ of $\g$ in \cite{J5} (not necessarily maximal). Also in this case, by \cite{FJ1}, \cite{FJ2} we already know that the algebra $Sy(\p)$ is polynomial. Moreover by \cite[Thm. 6.3]{JS} we know that the adapted pair for $\p_{\Lambda}$ provides a Weierstrass section for $Y(\p_{\Lambda})=Sy(\p)$. It follows that everything was already done for contractions of maximal parabolic subalgebras in type A too.

\subsection{}\label{compindex}
Now when $\g$ is simple of type ${\rm A}_{n-1}$ and when the standard Levi factor $\mathfrak r$ of the standard parabolic subalgebra $\p$ consists in two extremal blocks of the same size (size $s$) and one central block (of size $n-2s$), the contraction $\widetilde\p=\mathfrak r\ltimes\m^a$ does not coincide with $\p=\mathfrak r\oplus\m$ as a Lie algebra in general. In this article we focus on this case, which seems to be an instructive example for which  the general results obtained in \cite{F01} about the In\"on\"u-Wigner contractions $\widetilde\p$ may apply.

Since here $\h_{\Lambda}$ is of dimension one,  the canonical truncation $\widetilde\p_{\Lambda}$ of the contraction $\widetilde\p$ is either equal to the contraction $\widetilde{\p_{\Lambda}}=(\mathfrak r'\oplus\h_{\Lambda})\ltimes\m^a$ of the canonical truncation $\p_{\Lambda}$ of $\p$ or to the derived subalgebra $\widetilde\p'$ of the contraction $\widetilde\p$ (see subsection \ref{twoposs}). Using Ra\"\i s' formula \cite{Ra} for the index of a semi-direct product, we compute the index of the contraction $\widetilde{\p_{\Lambda}}$  and show that this index is equal to the index of $\p_{\Lambda}$ (see Lemma \ref{indexindex2}). We also compute the index of $\widetilde\p'$, showing that this index is equal to the index of $\p_{\Lambda}$ minus one if $3s\neq n$ (that is, when  the central block of the Levi factor of $\p$ is not of the same size as the two extremal blocks) and to the index of $\p_{\Lambda}$ plus one, if $3s=n$ (see Lemma \ref{indexindex1}).

\subsection{}

The rest of this paper (Section \ref{constructAP}) consists in the construction of an adapted pair for the contraction $\widetilde{\p_{\Lambda}}$ . 
Since in general an adapted pair for $\p_{\Lambda}$ does not provide an adapted pair for its contraction because too many zeroes may occur when we apply the coadjoint action of $\widetilde{\p_{\Lambda}}$ on the second element of such an adapted pair, we have to  arrange the roots associated to the contraction of $\p_{\Lambda}$ very carefully (see subsection \ref{ST}).  Under some further conditions (see Lemma \ref{lmAP}), an adapted pair for $\widetilde{\p_{\Lambda}}$ provides an upper bound for the formal character of the algebra of symmetric invariants $Y\big(\widetilde{\p_{\Lambda}}\big)$ : we verify that these further conditions are filled in subsection \ref{formch}. Moreover the lower bound for $Sy\big(\widetilde\p\big)$ constructed in \cite{F00} provides a lower bound for the formal character of $Y\big(\widetilde{\p_{\Lambda}}\big)$ (see subsection \ref{lowerbound}) and  this lower bound coincides with the mentioned upper bound (see Lemma \ref{caracequal}).  This implies by Lemma \ref{lmAP} that the algebra of symmetric invariants $Y\big(\widetilde{\p_{\Lambda}}\big)$ has a Weierstrass section and then is a polynomial algebra over $\mathbb K$ (see Thm \ref{thmWS}).
Since the generators of $Y\big(\widetilde{\p_{\Lambda}}\big)$ can be chosen as bi-homogeneous generators, the polynomiality of $Y\big(\widetilde{\p_{\Lambda}}\big)$ implies also the polynomiality of the subalgebra $S(\m)^{\mathfrak r'\oplus\h_{\Lambda}}\subset Y\big(\widetilde{\p_{\Lambda}}\big)$. Here as an $\mathfrak r$-module, $\m$ is isomorphic to $\mathbb K^a\otimes(\mathbb K^b)^*\oplus\mathbb K^b\otimes(\mathbb K^a)^*\oplus\mathbb K^a\otimes(\mathbb K^a)^*$ (with $a=s$ and $b=n-2s$) and $\mathfrak r'$ is isomorphic to $\mathfrak s\mathfrak l_a\times\mathfrak s\mathfrak l_b\times\mathfrak s\mathfrak l_a$.

When $3s\neq n$, the computation of indices mentioned in subsection \ref{compindex} allows to show that actually $\widetilde\p_{\Lambda}=\widetilde{\p_{\Lambda}}$ (see Lemma \ref{casdiff})  and then $Sy\big(\widetilde\p\big)=Y\big(\widetilde\p_{\Lambda}\big)=Y\big(\widetilde{\p_{\Lambda}}\big)$ which has then a Weierstrass section and is  polynomial (see Thm. \ref{thmWS}). We have also that $S(\m)^{\mathfrak r'}=S(\m)^{\mathfrak r'\oplus\h_{\Lambda}}$ in this case. 

When $3s=n$ that is, when the Levi factor of $\p$ consists in three blocks of the same size, we show in Lemma \ref{samesize} that $\widetilde\p_{\Lambda}$ is equal to $\widetilde\p'$ and then $Sy\big(\widetilde\p\big)=Y\big(\widetilde\p'\big)$. 
We cannot conclude for the moment  whether $Sy\big(\widetilde\p\big)=Y\big(\widetilde\p'\big)$ is or not polynomial in this case (see Remark \ref{polynomialitysamesize}).
\subsection{}
 In\"on\"u-Wigner contractions and their algebra of symmetric invariants were already extensively studied in \cite{P}, \cite{PPY1}, \cite{PPY2}, \cite{PY1}, \cite{PY2}, \cite{PY3}, \cite{Ph} in another context in general. Here in this present paper we believe that our results are new and our method is completely different from this used in the mentioned papers.

\section{Acknowledgements.}
We are very grateful to Oksana Yakimova for enlightening discussions about semi-direct products and also for some references. Her skills were also very helpful for computing some generic stabilizers (especially in Section \ref{compind}).

\section{Notation.}\label{Not}

Let $n$ be a positive integer. We denote by $\mathscr M_n(\mathbb K)$ the vector space  of square matrices of size $n$ with entries belonging to the base field $\mathbb K$. 
For $1\le i,\,j\le n$, denote by $E_{i,\,j}\in\mathscr M_n(\mathbb K)$ the elementary matrix with entry one on  row $i$ and  column $j$ (or simply, at the $(i,\,j)$ place) and zero elsewhere.
We denote by $\g\mathfrak l_n(\mathbb K)$ or simply $\g\mathfrak l_n$ the reductive Lie algebra equal to the vector space $\mathscr M_n(\mathbb K)$ and endowed with the usual Lie bracket $$[A,\,B]=AB-BA$$ for all $A,\,B\in\mathscr M_n(\mathbb K)$.
If $n\ge 2$, we also set $\g=\mathfrak s\mathfrak l_n(\mathbb K)$ (or simply $\mathfrak s\mathfrak l_n$) to be  the simple Lie subalgebra of $\g\mathfrak l_n(\mathbb K)$ consisting in matrices of trace zero (for $n=1$, we set $\mathfrak s\mathfrak l_1=\{0\}$).
 We take $\h$ to be the set of diagonal matrices in $\g$ : it is a Cartan subalgebra of $\g$. We denote by $\Delta$ the set of roots of $(\g,\,\h)$.
Recall that we have, as Lie algebras: 
\begin{equation}\g\mathfrak l_n=\mathfrak s\mathfrak l_n\oplus \mathbb K{\rm I}_n\label{decompgl}\end{equation}
where ${\rm I}_n=\sum_{i=1}^n E_{i,\,i}$ is the identity matrix. 

Recall the definition of the index given in subsection \ref{reg}. Then by  say \cite[1.9.12, 1.11.12]{D}, we have
\begin{align}    {\rm index}\;\mathfrak s\mathfrak l_n=n-1\;\;\;
{\rm and}\;\;\;\;\;  {\rm index}\;\mathfrak g\mathfrak l_n=n\label{indgl}.\end{align}

We denote by ${\rm GL}_n$ the group of invertible matrices of size $n$ with entries lying in $\mathbb K$ and by ${\rm SL}_n$ the subgroup of ${\rm GL}_n$ formed by matrices with determinant equal to one, so that
we have 
${\rm Lie}({\rm GL}_n)=\g\mathfrak l_n
\;\;{\rm and }\;\;\;
{\rm Lie}({\rm SL}_n)=\mathfrak s\mathfrak l_n.$

\subsection{}\label{notstand}
 We denote by $\varepsilon_j$, for $1\le j\le n$, the dual basis of the basis formed by the diagonal matrices $E_{i,\,i}$, for $1\le i\le n$ and we denote still by $\ep_j$ the restriction of $\ep_j$ to $\h$. We choose the set $\pi=\{\alpha_1,\,\ldots,\,\alpha_{n-1}\}$ as set of simple roots of $(\g,\,\h)$ with $\alpha_i=\varepsilon_i-\varepsilon_{i+1}\in\h^*$ for $1\le i\le n-1$. We denote by $(\,,\,)$ the positive definite symmetric bilinear form on $\bigoplus_{j=1}^n\mathbb Q\ep_j$ for which $(\ep_j)_{1\le j\le n}$ is an orthonormal basis. 

The coroot $\alpha_i^\vee$ associated with $\alpha_i$ is $\alpha_i^\vee=E_{i,\,i}-E_{i+1,\,i+1}$ and $$\h=\bigoplus_{1\le i\le n-1}\mathbb K\alpha_i^\vee.$$
This choice of simple roots allows to decompose the set of roots $\Delta$ as $$\Delta=\Delta^+\sqcup\Delta^-$$ with $\Delta^\pm=\Delta\cap(\pm\sum_{i=1}^{n-1}\mathbb N\alpha_i)$.

For $\alpha\in\Delta$, recall that $\alpha=\varepsilon_i-\varepsilon_j$ for some $1\le i\neq j\le n$. Then we set $x_{\alpha}=E_{i,\,j}$ : it is a root vector of $\g$ with weight $\alpha$ that is, for all $h\in\h$, we have
$$[h,\,x_{\alpha}]=\alpha(h)x_{\alpha}.$$

Identifying $\alpha$ and $x_{\alpha}$, we will say that the root $\alpha=\ep_i-\ep_j$ has row $i$ and column $j$.

Recall that we have the decomposition
$$\g=\h\oplus\bigoplus_{\alpha\in\Delta}\mathbb K x_{\alpha}.$$

For any simple root $\alpha\in\pi$, denote by $\varpi_{\alpha}$ the fundamental weight associated with it and with $\pi$ that is, $$\varpi_{\alpha}(\beta^\vee)=\delta_{\alpha,\,\beta}$$ for any $\alpha,\,\beta\in\pi$, where $\delta_{\alpha,\,\beta}$ is the Kronecker symbol. 

Let $\pi'\subset\pi$.
For any simple root $\alpha\in\pi'$, we denote by $\varpi'_{\alpha}$ the fundamental weight associated with it and with $\pi'$ that is, $$\varpi'_{\alpha}(\beta^\vee)=\delta_{\alpha,\,\beta}$$ for any $\alpha,\,\beta\in\pi'$.

Let $1\le i\le n-1$. We simply denote by $\varpi_i$, resp. $\varpi'_i$, the fundamental weight associated with $\pi$, resp. with $\pi'$, and with the simple root $\alpha_i$, resp. with $\alpha_i$ if $\alpha_i\in\pi'$.

We denote by $P(\pi)=\sum_{\alpha\in\pi}\mathbb Z\varpi_{\alpha}$ the weight lattice of $(\g,\,\h,\,\pi)$.

\subsection{}\label{rootsp}
From now on, choose $s$ a positive integer such that $2s<n$ and let $\p$ be the standard parabolic subalgebra  of $\g=\mathfrak s\mathfrak l_n$ associated with the subset of simple roots $\pi'=\pi\setminus\{\alpha_s,\,\alpha_{n-s}\}$. Set $a=s$ and $b=n-2s$. The standard Levi factor $\mathfrak r$ of $\p$ consists in two extremal  blocks which are symmetric with respect to the antidiagonal, then which are  of the same size $a$, and in one central block of size $b$. 

Set $$\Delta^\pm_{\pi'}=\Delta^\pm\cap(\pm \mathbb N\pi'),\;\n^\pm_{\pi'}=\bigoplus_{\alpha\in\Delta^\pm_{\pi'}}\mathbb K x_{\alpha},\;\m=\bigoplus_{\alpha\in\Delta^+\setminus\Delta^+_{\pi'}}\mathbb K x_{\alpha}.$$

Then the standard Levi factor $\mathfrak r$ of $\p$ is equal to:
$$\mathfrak r=\n^+_{\pi'}\oplus\h\oplus\n^-_{\pi'}$$ and $\m$ is the nilpotent radical of the parabolic subalgebra $\p$
so that we have:

\begin{equation}\p=\mathfrak r\oplus\m\label{decompp}\end{equation}
as $\mathfrak r$-modules.

The Levi subalgebra $\mathfrak r'$ of $\p$ (which is also equal to the derived subalgebra $[\mathfrak r,\,\mathfrak r]$ of $\mathfrak r$) is
$$\mathfrak r'=\n^+_{\pi'}\oplus\h_{\pi'}\oplus\n^-_{\pi'}$$
with $$\h_{\pi'}=\bigoplus_{\alpha\in\pi'}\mathbb K\alpha^{\vee}.$$

The derived subalgebra $\p'$ of $\p$ is equal to :

\begin{equation*}\p'=\mathfrak r'\oplus\m\end{equation*}  as $\mathfrak r'$-modules.

Setting $\m^-=\bigoplus_{\alpha\in\Delta^-\setminus\Delta^-_{\pi'}}\mathbb K x_{\alpha}$, the opposite parabolic subalgebra $\p^-$ of $\p$ is then equal to :
\begin{equation}\p^-=\mathfrak r\oplus\m^-\label{decomppmoins}\end{equation}
as $\mathfrak r$-modules.

Observe that, in our present case, $\dim\h_{\pi'}=n-3$.
We set $\h^{\pi\setminus\pi'}=\{h\in\h\mid\pi'(h)=0\}$, so that $\h=\h_{\pi'}\oplus\h^{\pi\setminus\pi'}$. The set of roots of $(\p,\,\h)$ is $\Delta(\pi'):=\Delta^+\sqcup\Delta^-_{\pi'}$.

We denote by $P(\pi')=\sum_{\alpha\in\pi'}\mathbb Z\varpi'_{\alpha}$ the weight lattice of $(\mathfrak r',\,\h_{\pi'},\,\pi')$.
By for instance \cite[2.5]{FJ2} there exists a positive integer $m$ such that
$$P(\pi)\subset P(\pi')\oplus\frac{1}{m}\sum_{\alpha\in\pi\setminus\pi'}\mathbb Z\varpi_{\alpha}.$$
Observe that, for every $\alpha\in\pi'$,  $\varpi'_{\alpha}$ is the projection of $\varpi_{\alpha}$ in $P(\pi')$ with respect to the above inclusion. If $\alpha\in\pi\setminus\pi'$, then the projection of $\varpi_{\alpha}$ in $P(\pi')$ is zero.

\subsection{}\label{weight}

As for instance in \cite[5.1]{FJ3}, we may define involutions $\bf i$ and $\bf j$  of $\pi$ as follows. Denoting by $w_0$, resp. $w_0'$, the longest element of the Weyl group $W$ of $(\g,\,\h)$, resp. of the Weyl group $W'$ of $(\mathfrak r',\,\h_{\pi'})$, we set ${\bf j}(\alpha)=-w_0(\alpha)$ for each $\alpha\in\pi$ and  ${\bf i}(\alpha)=-w_0'(\alpha)$, for each $\alpha\in\pi'$. For $\alpha\in\pi\setminus\pi'$, we set ${\bf i}(\alpha)={\bf j(ij)}^{r_{\alpha}}(\alpha)$ where $r_{\alpha}$ is the smallest nonnegative integer  such that ${\bf j(ij)}^{r_{\alpha}}(\alpha)\not\in\pi'$. In particular, if ${\bf j}(\alpha)\not\in\pi'$, then ${\bf i}(\alpha)={\bf j}(\alpha)$. 

In our present case ($\pi'=\pi\setminus\{\alpha_s,\,\alpha_{n-s}\}$ with $1\le s<n-s\le n-1$) we easily check that :
$$\begin{cases} {\bf  i}(\alpha_k)=\alpha_{s-k},\; {\bf i}(\alpha_{n-s+k})=\alpha_{n-k}&{\rm for}\; 1\le k\le s-1, \\
 {\bf i}(\alpha_k)=\alpha_{n-k}&{\rm for}\; s\le k\le n-s,\\
{\bf j}(\alpha_k)=\alpha_{n-k}&{\rm for}\; 1\le k\le n-1.
\end{cases}$$

We denote by $\langle{\bf ij}\rangle$  the subgroup of permutations in $\pi$ generated by the product $\bf ij$ and by $E(\pi')$ the set of $\langle {\bf ij}\rangle$-orbits in $\pi$.

We have 
\begin{align*}E(\pi')=\Big\{\{\alpha_k,\,\alpha_{n-s+k}\}\;;\;1\le k\le s-1,\;\,\{\alpha_{\ell}\};\; s\le \ell\le n-s\Big\}.\end{align*}

Recall the notation of subsection \ref{semi}. By \cite[Prop. 7.2]{FJ2} since $\g$ is of type A, we know that the algebra $Sy(\p)$ generated by symmetric semi-invariants in $S(\p)$  is a polynomial $\mathbb K$-algebra in $\lvert E(\pi')\rvert$ variables that is,  in $n-s$ variables, each of them having a weight $\delta_{O_{\alpha}}$ equal to 
\begin{equation}\delta_{O_{\alpha}}=d_{O_{\alpha}}+d_{{\rm\bf j}(O_{\alpha})}-d'_{O_{\alpha}}-d'_{{\rm\bf j}(O_{\alpha})}\label{poidsclas}\end{equation}
 for each $O_{\alpha}\in E(\pi')$, where $d_{O_{\alpha}}=\sum_{\gamma\in O_{\alpha}}\varpi_{\gamma}$ and $d'_{O_{\alpha}}=\sum_{\gamma\in O_{\alpha}\cap\pi'}\varpi'_{\gamma}$. 

By the above, the set  $\Lambda(\p)$ of weights of $Sy(\p)$ is the semigroup generated by the set $\{\delta_{O_{\alpha}};\;O_{\alpha}\in E(\pi')\}$. An easy computation shows that, for each $O_{\alpha}\in E(\pi')$, we have 
\begin{equation}\delta_{O_{\alpha}}=\sum_{j=1}^s\ep_j-\sum_{j=n+1-s}^n\ep_j=\varpi_s+\varpi_{n-s}\label{weightclassique}\end{equation}
by \cite[Planche I]{BOU}.

Then we have that
\begin{equation}\Lambda(\p)=\mathbb N(\varpi_s+\varpi_{n-s}).\label{weightSy(p)}\end{equation}

\subsection{}\label{truncCartan}
We set $\h_{\Lambda}=\p_{\Lambda}\cap\h^{\pi\setminus\pi'}$. Then we have that 
$$\h_{\pi'}\oplus\h_{\Lambda}=\p_{\Lambda}\cap\h.$$
In other words, in view of the definition of the canonical truncation (see Eq. \ref{deftrunc}), we have that $\h_{\pi'}\oplus\h_{\Lambda}$ is the orthogonal of $\mathbb K\Lambda(\p)$ in $\h$, with respect to the duality. Hence $\dim\h_{\Lambda}=1$.

 More precisely one has that $\p_{\Lambda}=\p'\oplus\h_{\Lambda}$ with $\h_{\Lambda}=\mathbb K \rm H_{\Lambda}$ given as follows.

\begin{equation} \rm H_{\Lambda}=\left[\begin{array}{ccc} {\rm I}_a&0&0\\
0&-\frac{2a}{b} {\rm I}_b&0\\
0&0&{\rm I}_a\\
\end{array}\right].
\end{equation}

Let us explain this below.
Since ${\rm tr}(\rm H_{\Lambda})=0$, one has that ${\rm H}_{\Lambda}\in\h$. Moreover write ${\rm H}_{\Lambda}=\sum_{i=1}^a E_{i,\,i}-\frac{2a}{b}\sum_{i=a+1}^{a+b}E_{i,\,i}+\sum_{i=a+b+1}^{2a+b} E_{i,\,i}$. One verifies that $\alpha_k(\rm H_{\Lambda})=0$ for all $1\le k\le n-1$, $k\not\in\{ s,\; n-s\}$. Hence ${\rm H}_{\Lambda}\in\h^{\pi\setminus\pi'}$. Finally one checks that $(\varpi_s+\varpi_{n-s})(\rm H_{\Lambda})=0$. Hence $\rm H_{\Lambda}\in\h_{\Lambda}$ and since $\dim\h_{\Lambda}=1$, we are done.

\subsection{}\label{defcontp}
 The In\"on\"u-Wigner (or one-parameter) contraction $\widetilde\p$ with respect to the decomposition (\ref{decompp}) of $\p$  is  the semi-direct product $\widetilde\p=\mathfrak r\ltimes\m^a$ as defined in subsection \ref{contr}. This means that, 
as an $\mathfrak r$-module, $\widetilde\p$ is isomorphic to $\p$ and as a Lie algebra, $\widetilde\p$ has $\m^a$ as an abelian ideal.

\subsection{}
Following the definition in subsection \ref{contr}, we may also consider the In\"on\"u-Wigner contraction $\widetilde{\p_{\Lambda}}$ of $\p_{\Lambda}$ with respect to the decomposition $\p_{\Lambda}=(\mathfrak r'\oplus\h_{\Lambda})\oplus\m$, so that we have
$$\widetilde{\p_{\Lambda}}=(\mathfrak r'\oplus\h_{\Lambda})\ltimes\m^a$$
and the In\"on\"u-Wigner contraction $\widetilde{\p'}$ of $\p'$ with respect to the decomposition $\p'=\mathfrak r'\oplus\m$, so that $\widetilde{\p'}=\mathfrak r'\ltimes\m^a$.
We can easily check that the contraction $\widetilde{\p'}$ of $\p'$ coincides with the derived subalgebra of the contraction $\widetilde\p$ of $\p$. That is why we can denote it also by $\widetilde\p'$.

\section{Computation of indices.}\label{compind}

We keep the hypotheses and notation of Section \ref{Not} and we show in this section that the index of the contraction $\widetilde{\p_{\Lambda}}$ is equal to ${\rm index}\,{\p_{\Lambda}}$, and that the index of $\widetilde{\p}'$ is equal to ${\rm index}\,{\p_{\Lambda}}-1$ when $3s\neq n$ and is equal to ${\rm index}\,{\p_{\Lambda}}+1$ when $3s=n$. 

\subsection{}
Firstly, we may recall that in the present case, we have 

\begin{lm}\label{indexp}
\begin{equation}{\rm index}\,\p_{\Lambda}=n-s.\label{indexcl}\end{equation}
\end{lm}

\begin{proof}
 By subsection \ref{GKdim} we know that ${\rm index}\,\p_{\Lambda}$ is equal to the Gelfand-Kirillov dimension of $Sy(\p)$, which is also equal by \cite[Cor. 5.4.2, Prop. 7.1]{FJ2} to the number $\lvert E(\pi')\rvert$ of $\langle{\bf ij}\rangle$-orbits in $\pi$. 
By subsection \ref{weight}, we know that $\lvert E(\pi')\rvert=n-s$. Hence we obtain Eq. \ref{indexcl}.
\end{proof}

\subsection{}
To compute the index of a semi-direct product $\q=\mathfrak l\ltimes V^a$, where $\mathfrak l$ is a finite-dimensional Lie algebra and $V$ is a finite-dimensional $\mathfrak l$-module, and also an abelian ideal in $\q$, we use the following Ra\"\i s' formula (\cite{Ra} or \cite[40.4]{TY}).

We view each element of $V^*$ as an element in $\q^*$ which vanishes on $\mathfrak l$ and each element of $\mathfrak l^*$ as an element in $\q^*$ which vanishes on $V$. Then $\q^*=\mathfrak l^*\oplus V^*$ as vector spaces.
For $\gamma\in V^*$, let $\mathfrak l_{\gamma}=\{x\in \mathfrak l\mid \forall v\in V, \,\gamma(x.v)=0\}$ be the stabilizer of $\gamma$ in $\mathfrak l$. We have :

\begin{equation*} {\rm index}\,\q={\rm index}\,\mathfrak l_{\gamma}+\dim V-\dim \mathfrak l+\dim \mathfrak l_{\gamma}\end{equation*}
where $\gamma\in V^*$ is a generic point with respect to $\q^*$ (that is, there exists an element $\xi\in\mathfrak l^*$ such that $\xi+\gamma\in\q^*$ is a regular element of $\q^*$, in the sense of subsection \ref{reg}).

Then for the semi-direct product $\widetilde\p'=\mathfrak r'\ltimes\m^a$, the Ra\"\i s' formula gives
\begin{equation}{\rm index}\,\widetilde\p'={\rm index}\,\mathfrak r'_{\gamma}+\dim\m-\dim \mathfrak r'+\dim \mathfrak r'_{\gamma}\label{index1}\end{equation}
and for the semi-direct product $\widetilde{\p_{\Lambda}}=(\mathfrak r'\oplus\h_{\Lambda})\ltimes\m^a$, the Ra\"\i s' formula gives
\begin{equation}{\rm index}\,\widetilde{\p_{\Lambda}}={\rm index}\,(\mathfrak r'\oplus\h_{\Lambda})_{\gamma}+\dim\m-\dim(\mathfrak r'\oplus\h_{\Lambda})+\dim(\mathfrak r'\oplus\h_{\Lambda})_{\gamma}\label{index2}\end{equation}
for a generic point $\gamma\in\m^*$ (with respect to $\widetilde\p'^*$, resp. to $\widetilde{\p_{\Lambda}}^*$).

\subsection{}
\begin{lm}\label{indexindex1}
We have that

\begin{align}{\rm index}\,\widetilde\p'=\begin{cases}{\rm index}\,\p_{\Lambda}-1&{\rm if}\;3s\neq n\\
{\rm index}\,\p_{\Lambda}+1&{\rm if}\;3s=n\end{cases}\label{ind1}\end{align}

\end{lm}

\begin{proof}

According to Ra\"\i s' formula, we have first to compute  the expression of a generic point in $\m^*$. 

Recall that we have set $a=s$ and $b=n-2s$ which are respectively the sizes of the two extremal symmetric blocks and of the central block of the Levi factor of $\p$. 
Set $R=({\rm GL}_a\times{\rm GL}_b\times{\rm GL}_a)\cap{\rm SL}_n$ so that $\mathfrak r={\rm Lie}(R)$ and denote by $\exp(\m^a)$ the group generated by the matrices $\exp(x)$ for $x\in\m^a$. Set also $\widetilde P=R\ltimes{\rm exp}(\m)$  the semi-direct product where $\exp(\m)$ is a normal Abelian unipotent subgroup of $\widetilde P$.
Then $\widetilde\p={\rm Lie}(\widetilde P)$.

 Identifying, through the Killing form of $\g$, the subspace $\m^*$ with the nilpotent radical $\m^-$ of the opposite parabolic subalgebra $\p^-$ of $\p$ (see Eq. \ref{decomppmoins}), we have that 
 \begin{equation}\m^*\simeq(\mathbb K^a)^*\otimes\mathbb K^b\oplus(\mathbb K^a)^*\otimes\mathbb K^a\oplus(\mathbb K^b)^*\otimes\mathbb K^a\label{m}\end{equation} as an $R$-module.
 Let $\gamma\in\m^*$ (or in $\m^-$ via the above identification). One says that $\gamma$ is $\mathfrak r$-regular if $\dim\mathfrak r_{\gamma}=\min_{\delta\in\m^*}\dim\mathfrak r_{\delta}$. By \cite[40.2.1]{TY}, the set of $\mathfrak r$-regular elements of $\m^*$ is a non-empty Zariski open subset of $\m^*$. 
  Moreover since the projection of ${\mathfrak r}^*\oplus\m^*$ onto $\m^*$ is an open map, the set of elements in $\m^*$ (or in $\m^-$) which are generic points with respect to $\widetilde\p^*$ is also a non-empty open subset of $\m^*$ (or of $\m^-$). It is the same if we consider generic points in $\m^*$ with respect to $\widetilde\p'^*$ and to $\widetilde{\p_{\Lambda}}^*$.
 
 Then the set of elements in $\m^-$ which are $\mathfrak r$-regular and also generic points with respect to $\widetilde\p'^*$ and to $\widetilde{\p_{\Lambda}}^*$ is a non-empty open subset of $\m^-$.
 
 Denote by $\mathscr M_{i,\,j}(\mathbb K)$ the set of all matrices with $i$ rows and $j$ columns and whose entries belong to $\mathbb K$. Denote by $0_{i,\,j}\in\mathscr M_{i,\,j}(\mathbb K)$ the matrix with zero everywhere.
  
  Consider the subset $\m_1$ of $\m^-$ formed by the following elements $\gamma$ :
 \[
 \gamma=\left[
 \begin{array}{c|c|c}
 0_{a,\,a}&0_{a,\,b}&0_{a,\,a}\\
 \hline
 \rm X&0_{b,\,b}&0_{b,\,a}\\
 \hline
 \rm Z&\rm Y&0_{a,\,a}
 \end{array}
 \right]
 \]
 with ${\rm X}\in\mathscr M_{b,\,a}(\mathbb K)\simeq(\mathbb K^a)^*\otimes\mathbb K^b$, ${\rm Y}\in\mathscr M_{a,\,b}(\mathbb K)\simeq(\mathbb K^b)^*\otimes\mathbb K^a$ and ${\rm Z}\in{\rm GL}_a$. 
 
 The subset $\m_1$ is a non-empty open subset of $\m^-$ and then the set $\m_2$ of elements $\gamma\in\m_1$  which are also $\mathfrak r$-regular and a generic point with respect to $\widetilde\p'^*$ and to $\widetilde{\p_{\Lambda}}^*$ is a non-empty open subset of $\m^-$.
 
 Denote by $\m_3$ the  subset of $\m^-$ formed by the following elements $\widetilde\gamma$ :
 \[
\widetilde\gamma= \left[
 \begin{array}{c|c|c}
 0_{a,\,a}&0_{a,\,b}&0_{a,\,a}\\
 \hline
 \rm\widetilde X&0_{b,\,b}&0_{b,\,a}\\
 \hline
 {\rm I}_a&\rm\widetilde Y&0_{a,\,a}
 \end{array}
 \right]
 \]
 where $({\rm\widetilde X},\,{\rm\widetilde Y})\in(\mathbb K^a)^*\otimes\mathbb K^b\times(\mathbb K^b)^*\otimes\mathbb K^a$. 
 
 Observe that each element $\widetilde\gamma\in\m_3$ is $R$-conjugate to an element $\gamma\in\m_1$. The subset $\{\widetilde\gamma\in\m_3\mid \exists \gamma\in\m_2,\,\widetilde\gamma\in R.\gamma\}$  is a non-empty open subset of $\m_3$ and the irreducible smooth variety $\m_3$ has $(\mathbb K^a)^*\otimes\mathbb K^b\oplus(\mathbb K^b)^*\otimes\mathbb K^a$ as a tangent space.

 Consider
 an element $g\in R$ which stabilizes $\m_3$. Then $g$ is of the form : 
  \[g=\left[
 \begin{array}{c|c|c}
\rm A&0_{a,\,b}&0_{a,\,a}\\
 \hline
 0_{b,\,a}&\rm B&0_{b,\,a}\\
 \hline
 0_{a,\,a}&0_{a,\,b}&\rm A\\
 \end{array}
 \right]
 \]
 
 with ${\rm A }\in{\rm GL}_a$, ${\rm B}\in{\rm GL}_b$ such that $(\det\rm A)^2\det\rm B=1$ and we have :
 
  \begin{align}
g.\widetilde\gamma=g\widetilde\gamma g^{-1}= \left[
 \begin{array}{c|c|c}
 0_{a,\,a}&0_{a,\,b}&\hskip0.5cm 0_{a,\,a}\hskip0.7cm\\
 \hline
 \rm B\widetilde X \rm A^{-1}&0_{b,\,b}&\hskip0.5cm 0_{b,\,a}\hskip0.7cm\\
 \hline
 {\rm I}_a&\rm A\widetilde Y \rm B^{-1}&\hskip0.5cm 0_{a,\,a}\hskip0.7cm
 \end{array}
 \right]
 \label{act}\end{align}
 
The above action of ${\rm GL}_a\times{\rm GL}_b$ on $\mathfrak m_3$ comes from the symmetric pair $(\g\mathfrak l_{a+b},\,\g\mathfrak l_a\times\mathfrak g\mathfrak l_b)$.

Let 
\[ {\rm C}={\rm diag}(d_1,\,\ldots,\,d_c)\] be a diagonal matrix of size $c$ 
with $c=\min(a,\,b)$, $d_i\in\mathbb K^*$ for all $i$, $1\le i\le c$, and $d_i\neq d_j$ for all $i\neq j$. Denote by $\mathcal C$ the set of all such matrices ${\rm C}$.

\begin{itemize}

\item[] Case I : $a>b$. For $\rm C\in\mathcal C$, one sets

\[
\rm Z_C=\left [
\begin{array}{c|c|c}
0_{a,\,a}&0_{a,\,b}&\hskip0.4cm 0_{a,\,a}\hskip1cm\\
\hline 
\left[\begin{array}{c|c}{\rm C}&0_{b,\,a-b}\end{array}\right]&0_{b,\,b}&\hskip0.5cm 0_{b,\,a}\hskip1cm\\
\hline
{\rm I}_a&
\left[\begin{array}{c}{\rm I}_b\\
\hline
0_{a-b,\,b}
\end{array}\right]& \hskip0.4cm 0_{a,\,a} \hskip1cm\\
\end{array}\right ]
\]

\item[] Case II : $a<b$. For $\rm C\in\mathcal C$, one sets

\[
\rm Z_C=\left [
\begin{array}{c|c|c}
0_{a,\,a}&0_{a,\,b}&\hskip0.5cm 0_{a,\,a}\hskip0.7cm\\
\hline
\left[\begin{array}{c}{\rm C}\\\hline
0_{b-a,\,a}\end{array}\right]&0_{b,\,b}&\hskip0.5cm 0_{b,\,a}\hskip0.7cm\\
\hline
{\rm I}_a&
\left[\begin{array}{c|c}{\rm I}_a&
0_{a,\,b-a}\\
\end{array}\right]&\hskip0.5cm 0_{a,\,a}\hskip0.7cm\\
\end{array}\right]
\]

\item[] Case III : $a=b$. For $\rm C\in\mathcal C$, one sets

\[
\rm Z_C=\left[
\begin{array}{c|c|c}
0_{a,\,a}&0_{a,\,a}& 0_{a,\,a}\\
\hline
{\rm C}
&0_{a,\,a}& 0_{a,\,a}\\
\hline
{\rm I}_a&
{\rm I}_a
&0_{a,\,a}\\
\end{array}\right]
\]

\end{itemize}
\medskip

The theory of symmetric pairs and in particular of Cartan subspaces in the odd part  of the symmetric pair $(\g\mathfrak l_{a+b},\,\g\mathfrak l_a\times\mathfrak g\mathfrak l_b)$ (see for instance \cite[37.4]{TY}) allows us to claim that there exists a generic point in $\m^*$ with respect to $\widetilde\p'^*$ and to $\widetilde{\p_{\Lambda}}^*$, which is also $\mathfrak r$-regular and which is $R$-conjugate to an element $\rm Z_C$, with $\rm C\in\mathcal C$, as above.
For the reader's convenience, we will give some explanations below.
Set ${\rm G}_0={\rm GL}_a\times{\rm GL}_b$, $\mathcal Z=\{{\rm Z_C};\;{\rm C}\in\mathcal C\}\subset\m_3$ and $\mu:{\rm G}_0\times\mathcal Z\longrightarrow{\rm G}_0.\mathcal Z$ the action given by Eq. \ref{act}. Recall that $c=\min(a,\,b)$, so that $\dim\mathcal Z=c$.  For ${\rm C}\in\mathcal C$, we can easily check that the dimension of the fibre $\mu^{-1}(\rm Z_C)$ is equal to $c+(a-b)^2$  and then by \cite[Cor. 15.5.4]{TY} that 

\begin{align*}\dim({\rm G}_0.\mathcal Z)\ge \dim({\rm G}_0\times\mathcal Z)-\dim\mu^{-1}({\rm Z_C})\\
=a^2+b^2+c-(c+(a-b)^2)=2ab=\dim\m_3.\end{align*} 
Hence we have that the Zariski closure of ${\rm G}_0.\mathcal Z$ is equal to $\m_3$.
Then ${\rm G}_0.\mathcal Z$, which is a constructible set, as the image of the constructible set ${\rm G}_0\times\mathcal Z$ by a morphism, contains a non-empty open subset of its closure that is, of $\m_3$.

We can conclude by what we said above that there exists an element $\widetilde\gamma\in\m_3$ which is ${\rm G}_0$-conjugate to an element $\rm Z_C$ for ${\rm C}\in\mathcal C$ and which is $R$-conjugate to an element in $\m^*$ which is $\mathfrak r$-regular and also generic with respect to $\widetilde\p'^*$ and to $\widetilde{\p_{\Lambda}}^*$.

It remains to give the expression of an element $g\in R$ in the stabilizer $R_{\rm Z_C}$ of $\rm Z_C$ for ${\rm C}\in\mathcal C$. For any $r\in\mathbb N^*$, denote by $\mathbb T_r\simeq(\mathbb K^*)^r$ the standard maximal torus of ${\rm GL}_r$. Then ${\rm Lie}\,(\mathbb T_r)$ is the set of all diagonal matrices of $\g\mathfrak l_r(\mathbb K)\simeq\mathscr M_r(\mathbb K)$. An easy computation shows that an element $g$ belongs to $R_{\rm Z_C}$ if and only if $g$ has one of the  three following forms.\medskip

\begin{enumerate}

\item[]Case I : $a>b$.

 \[g=\left[
 \begin{array}{c|c|c}
\left[\begin{array}{c|c}
{\rm T}&0_{b,\,a-b}\\
\hline
0_{a-b,\,b}&{\rm D }\\
\end{array}\right]&0_{a,\,b}&0_{a,\,a}\\
\hline
 0_{b,\,a}& 
 {\rm T}&0_{b,\,a}\\\hline
 0_{a,\,a}& 0_{a,\,b}&\left[\begin{array}{c|c}
{\rm T}&0_{b,\,a-b}\\
\hline
0_{a-b,\, b}&{\rm D} \\
\end{array}\right]\\
 \end{array}
 \right]
 \]

with ${\rm T}\in\mathbb T_b,\;{\rm D}\in{\rm GL}_{a-b}$ and $(\det {\rm T})^3(\det {\rm D})^2=1$.

\medskip

\item[] Case II : $a<b$.

 \[g=\left[
 \begin{array}{c|c|c}
 {\rm T}&0_{a,\,b}&0_{a,\,a}\\\hline
 0_{b,\,a}&\left[\begin{array}{c|c}
{\rm T}&0_{a,\,b-a}\\
\hline
0_{b-a,\,a}&{\rm D }\\
\end{array}\right]
&0_{b,\,a}\\
\hline
0_{a,\,a}&0_{a,\,b}&{\rm T} \\
 \end{array}
 \right]
 \]

with ${\rm T}\in\mathbb T_a;\;{\rm D}\in{\rm GL}_{b-a}$ and $(\det {\rm T})^3\det {\rm D}=1$.

\medskip

\item[] Case III : $a=b$.

 \[g=\left[
 \begin{array}{c|c|c}
 {\rm T}&0_{a,\,a}&0_{a,\,a}\\\hline
 0_{a,\,a}&{\rm T}&0_{a,\,a}\\
\hline
0_{a,\,a}&0_{a,\,a}&{\rm T}\\
 \end{array}
 \right]
 \]

with ${\rm T}\in\mathbb T_a$ and $(\det \rm T)^3=1$.
\end{enumerate}

Denote by $R'$ the derived subgroup of $R$ so that $R'={\rm Lie}(\mathfrak r')$. Then $g\in R'_{\rm Z_C}$  if and only if $g$ is of the above form (case I, II or III) with moreover $\det\rm T=\det\rm D=1$.

Let $\gamma\in\m^*$ which is $\mathfrak r$-regular, generic with respect to $\widetilde\p'^*$ and to $\widetilde{\p_{\Lambda}}^*$ and which is $R$-conjugate to an element $\rm Z_C$, with $\rm C\in\mathcal C$, as above. Such an element $\gamma$ exists by what we explained above.

Then the above expression of an element in $R'_{\rm Z_C}$ gives that \begin{equation*}\dim\mathfrak r'_{\gamma}=\dim R'_{\gamma}=\dim R'_{\rm Z_C}=
\begin{cases}
b-1+(a-b)^2-1 &\hbox{\rm if}\;a>b\\
a-1+(b-a)^2-1&\hbox{\rm if}\;a<b\\
a-1&\hbox{\rm if}\;a=b
\end{cases}
\end{equation*}

and moreover by subsection \ref{reg} and by Eq. \ref{indgl} we have \begin{equation*}
{\rm index}\,\mathfrak r'_{\gamma}={\rm index}\, R'_{\gamma}={\rm index}\, R'_{\rm Z_C}=
\begin{cases}
b-1+a-b-1=a-2 &\hbox{\rm if}\;a>b\\
a-1+b-a-1=b-2&\hbox{\rm if}\;a<b\\
a-1&\hbox{\rm if}\;a=b
\end{cases}
\end{equation*}

One also has that
$$\dim\m=2ab+a^2;\;\dim\mathfrak r'=2(a^2-1)+b^2-1$$ by Eq. \ref{m} and since $\mathfrak r'\simeq\mathfrak s\mathfrak l_a\times\mathfrak s\mathfrak l_b\times\mathfrak s\mathfrak l_a$.

Then, by Eq. \ref{index1}, one has that
\begin{equation*} {\rm index}\,\widetilde\p'=\begin{cases}
a+b-1&\hbox{\rm if}\;a>b\\
a+b-1&\hbox{\rm if}\;a<b\\
2a+1&\hbox{\rm if}\;a=b
\end{cases}.  \end{equation*}

Recalling that $a=s$ and $b=n-2s$ and Eq. \ref{indexcl}, this completes the proof.\end{proof}

\subsection{}

\begin{lm}\label{indexindex2}

We have that
 \begin{equation}{\rm index}\,\widetilde{\p_{\Lambda}}={\rm index}\,\p_{\Lambda}.\label{ind2}\end{equation}

\end{lm}
\begin{proof}
To compute now the index of $\widetilde{\p_{\Lambda}}$, recall subsection \ref{truncCartan}, notably the expression of $\h_{\Lambda}=\mathbb K{\rm H}_{\Lambda}$.

We may observe that $\rm H_{\Lambda}$ does not belong to the generic stabilizer $\mathfrak r_{\rm Z_C}={\rm Lie}(R_{\rm Z_C})$ (see cases I, II and III in the end of the proof of the previous lemma). 

Recall Eq. \ref{decompgl}.
Consider case I, when $a>b$.
By what we said in the end of the proof of the previous lemma, the generic stabilizer $\mathfrak r_{\rm Z_C}\subset\mathfrak s\mathfrak l_n$ consists of the following matrices

 \[\left[
 \begin{array}{c|c|c}
\left[\begin{array}{c|c}
t'+y{\rm I}_b&0_{b,\,a-b}\\
\hline
0_{a-b,\,b}&d\\
\end{array}\right]&0_{a,\,b}&0_{a,\,a}\\
\hline
 0_{b,\,a}&
 t'+y{\rm I}_b&0_{b,\,a}\\\hline
 0_{a,\,a}&0_{a,\,b}&\left[\begin{array}{c|c}
t'+y{\rm I}_b&0_{b,\,a-b}\\
\hline
0_{a-b,\, b}&d \\
\end{array}\right]\\
 \end{array}
 \right]
 \]
for every $t'\in{\rm Lie}(\mathbb T_b)\cap\mathfrak s\mathfrak l_b$, $d\in \g\mathfrak l_{a-b}$ and $y\in\mathbb K$ such that ${\rm tr}(d)=-\frac{3b}{2}y$.

Write \[\left[\begin{array}{c|c}
t'+y{\rm I}_b&0_{b,\,a-b}\\
\hline
0_{a-b,\,b}&d\\
\end{array}\right]\in\g\mathfrak l_a\]  as $\zeta+z{\rm I}_a$ with $\zeta\in \mathfrak s\mathfrak l_a$ and $z\in\mathbb K$  by Eq. \ref{decompgl} with $2az+by=0$.

We then obtain that every element in $\mathfrak r_{\rm Z_C}$ can be written as
$$(\zeta,\,t',\,\zeta)+(z{\rm I}_a, y{\rm I}_b,\,z{\rm I}_a)$$
with 
\[\begin{cases}(\zeta,\,t',\,\zeta)\in\mathfrak s\mathfrak l_a\times\mathfrak s\mathfrak l_b\times\mathfrak s\mathfrak l_a=\mathfrak r'\\
 (z{\rm I}_a, y{\rm I}_b,\,z{\rm I}_a)=-\frac{b}{2a}y{\rm H}_{\Lambda}\in\h_{\Lambda}.
\end{cases}\]

Hence in this case, we have $(\mathfrak r'\oplus\h_{\Lambda})_{\rm Z_C}=\mathfrak r_{\rm Z_C}$ which is isomorphic to the space of matrices  $(t,\,d)\in{\rm Lie}(\mathbb T_b)\times \g\mathfrak l_{a-b}$ such that $2{\rm tr}(d)+3{\rm tr}(t)=0$.
Then
\begin{align*}\dim(\mathfrak r'\oplus\h_{\Lambda})_{\rm Z_C}=&\dim ({\rm Lie}(\mathbb T_b))+\dim\g\mathfrak l_{a-b}-1 \\
=&b+(a-b)^2-1=a^2+b^2-2ab+b-1\\
{\rm and \;\;\;\;index}\,(\mathfrak r'\oplus\h_{\Lambda})_{\rm Z_C}=&{\rm index}\,({\rm Lie}(\mathbb T_b))+{\rm index}\,\g\mathfrak l_{a-b}-1\\
=&b+a-b-1=a-1\;\;\hbox{\rm by Eq. \ref{indgl}.}\\
\end{align*} 

We obtain in this case, by Eq. \ref{index2}, that :

\begin{align*}
&{\rm index}\,\widetilde{\p_{\Lambda}}\\
=&a-1+2ab+a^2-(2(a^2-1)+b^2-1+1)+a^2+b^2-2ab+b-1\\
=&a+b=s+n-2s=n-s={\rm index}\,\p_{\Lambda}\;{\rm by\; Eq.\; \ref{indexcl}.}\\
\end{align*}

Consider case II, when $a<b$. Similarly as in the previous case, we obtain  that every element in $\mathfrak r_{\rm Z_C}$ can be written as
$$(t',\,\zeta,\,t')+\big(y{\rm I}_a,\,z{\rm I}_b,\,y{\rm I}_a\big)$$
with $y\in\mathbb K$, $z=-\frac{2a}{b}y$
and $(t',\,\zeta,\,t')\in\mathfrak s\mathfrak l_a\times\mathfrak s\mathfrak l_b\times\mathfrak s\mathfrak l_a=\mathfrak r'$.
Hence  $(y{\rm I}_a, z{\rm I}_b,\,y{\rm I}_a)=y{\rm H}_{\Lambda}\in\h_{\Lambda}$
and we have $(\mathfrak r'\oplus\h_{\Lambda})_{\rm Z_C}=\mathfrak r_{\rm Z_C}$ which is isomorphic to the space of matrices  $(t,\,d)\in{\rm Lie}(\mathbb T_a)\times \g\mathfrak l_{b-a}$ such that ${\rm tr}(d)+3{\rm tr}(t)=0$.
Then

\begin{align*}\dim(\mathfrak r'\oplus\h_{\Lambda})_{\rm Z_C}=&\dim ({\rm Lie}(\mathbb T_a))+\dim\g\mathfrak l_{b-a}-1 \\
=&a+(b-a)^2-1=a^2+b^2-2ab+a-1\\
{\rm and\;\;\;index}\,(\mathfrak r'\oplus\h_{\Lambda})_{\rm Z_C}=&{\rm index}\,({\rm Lie}(\mathbb T_a))+{\rm index}\,\g\mathfrak l_{b-a}-1\\
=&a+b-a-1=b-1\;\;\hbox{\rm by Eq. \ref{indgl}}.\\
\end{align*} 

We obtain in this case, by Eq. \ref{index2}, that :
\begin{align*}
&{\rm index}\,\widetilde{\p_{\Lambda}}\\
=&b-1+2ab+a^2-(2(a^2-1)+b^2-1+1)+a^2+b^2-2ab+a-1\\
=&a+b=s+n-2s=n-s={\rm index}\,\p_{\Lambda}\;{\rm by\; Eq. \;\ref{indexcl}}.\\
\end{align*}

Finally for case III, when $a=b$, the generic stabilizer $\mathfrak r_{\rm Z_C}$ consists of the following matrices 

\[\left[
 \begin{array}{c|c|c}
 t'+y{\rm I}_a&0_{a}&0_a\\\hline
 0_{a}&t'+y{\rm I}_a&0_a\\
\hline
0_{a}&0_{a}&t'+y{\rm I}_a\\
 \end{array}
 \right]
 \]
for every $t'\in{\rm Lie}(\mathbb T_a)\cap\mathfrak s\mathfrak l_a$ and $y\in\mathbb K$, with $3ay=0$. Hence $y=0$ and we have 
$(\mathfrak r'\oplus\h_{\Lambda})_{\rm Z_C}=\mathfrak r_{\rm Z_C}=\mathfrak r'_{\rm Z_C}$ which is isomorphic to the space ${\rm Lie}(\mathbb T_a)\cap\mathfrak s\mathfrak l_a$.

Hence $\dim(\mathfrak r'\oplus\h_{\Lambda})_{\rm Z_C}={\rm index}(\mathfrak r'\oplus\h_{\Lambda})_{\rm Z_C}=a-1$

and by Eq. \ref{index2}, we have :
\begin{align*}
{\rm index}\,\widetilde{\p_{\Lambda}}=&a-1+3a^2-(3(a^2-1)+1)+a-1\\
=&2a=2s=n-s={\rm index}\,\p_{\Lambda}\;{\rm by\; Eq. \;\ref{indexcl}.}\\
\end{align*}  
This completes the proof.
\end{proof}

\section{The canonical truncation of the contraction.}
We continue with the same hypotheses and notation as in the previous Section. 

\subsection{}\label{twoposs}
Recall that $\widetilde\p_{\Lambda}$ denotes the canonical truncation of the contraction $\widetilde\p=\mathfrak r\ltimes\m^a$ of $\p=\mathfrak r\oplus\m$ and that $\widetilde\p'=\mathfrak r'\ltimes\m^a$. Recall also that $\widetilde{\p_{\Lambda}}=(\mathfrak r'\oplus\h_{\Lambda})\ltimes\m^a$ denotes the contraction of $\p_{\Lambda}=(\mathfrak r'\oplus\h_{\Lambda})\oplus\m$. We have that
$\widetilde\p_{\Lambda}=\widetilde\p'\oplus\widetilde\h_{\Lambda}$ where
$\widetilde\h_{\Lambda}$ is a vector subspace of $\h_{\Lambda}$ since $\widetilde\p_{\Lambda}$ is the canonical truncation ${(\widetilde{\p_{\Lambda}})}_{\Lambda}$ of the contraction $\widetilde{\p_{\Lambda}}$ by \cite[Lem. 2.6.4]{F01} and since $\widetilde\p'\subset\widetilde\p_{\Lambda}$.
Hence only two possibilities may occur :
$$\widetilde\p_{\Lambda}=\widetilde{\p_{\Lambda}}=\widetilde\p'\oplus\h_{\Lambda} \;{\rm or}\;\;\widetilde\p_{\Lambda}=\widetilde\p'$$  since moreover $\dim\h_{\Lambda}=1$.
\subsection{}
Lemma \ref{indexindex1} implies, when $3s\neq n$, namely when the Levi factor of $\p$ consists in three symmetric blocks, with the central block not of the same size as the two extremal blocks, that the canonical truncation $\widetilde\p_{\Lambda}$ of the contraction $\widetilde\p$ is equal to the contraction $\widetilde{\p_{\Lambda}}$.

\begin{lm}\label{casdiff}

When $3s\neq n$, we have:
$$\widetilde\p_{\Lambda}=\widetilde{\p_{\Lambda}}.$$ 
\end{lm}

\begin{proof}
By \cite[Lem. 2.6.4 and Proof of Lem. 3.6.1, Eq. (42)]{F01} we have:
\begin{equation*} \dim\widetilde\p_{\Lambda}+{\rm index}\,\widetilde\p_{\Lambda}=\dim\widetilde{\p_{\Lambda}}+{\rm index}\,\widetilde{\p_{\Lambda}}\end{equation*}
Moreover the index cannot decrease under contraction, see \cite[Sec. 4]{Y1}. 
Then one has that:
\begin{equation*} {\rm index}\,\widetilde\p_{\Lambda}\ge {\rm index}\,\p_{\Lambda}+\dim\widetilde{\p_{\Lambda}}-\dim\widetilde\p_{\Lambda}\end{equation*}
By \cite[Lem. 2.6.4]{F01} one has that:
$$\widetilde\p_{\Lambda}={(\widetilde{\p_{\Lambda}})}_{\Lambda}\subset \widetilde{\p_{\Lambda}}.$$ 
It follows that $${\rm index}\,\widetilde\p_{\Lambda}\ge{\rm index}\,\p_{\Lambda}.$$
Hence we cannot have that $\widetilde\p_{\Lambda}=\widetilde\p'$ when $3s\neq n$
by Lemma \ref{indexindex1}. Then subsection \ref{twoposs} completes the proof in this case.
\end{proof}

\section{Construction of an adapted pair for the contraction of the canonical truncation.}\label{constructAP}

We continue with the same hypotheses and notation as in the previous Section  (we do not assume that $3s\neq n$). Here we will construct an adapted pair for the In\"on\"u-Wigner contraction $\widetilde{\p_{\Lambda}}$ of $\p_{\Lambda}$ and we will deduce that the algebra of invariants $Y\big(\widetilde{\p_{\Lambda}}\big)$ is a polynomial algebra over $\mathbb K$.

\subsection{}
To construct an adapted pair for the contraction $\widetilde{\p_{\Lambda}}=(\mathfrak r'\oplus\h_{\Lambda})\ltimes\m^a$ of $\p_{\Lambda}$, we will decompose the set of roots $\Delta(\pi')=\Delta^+\sqcup\Delta^-_{\pi'}$ of $(\p,\,\h)$ (see subsection \ref{rootsp}) into socalled Heisenberg sets. 

Recall some results in \cite{F01}.

\begin{defi}{(\cite[Def. 4.1.1]{F01})} Let $\gamma\in\Delta$. An Heisenberg set $\Gamma_{\gamma}$ with centre $\gamma$ is a subset of $\Delta$ containing $\gamma$ verifying that, for every $\alpha\in\Gamma_{\gamma}\setminus\{\gamma\}$, there exists $\alpha'\in\Gamma_{\gamma}\setminus\{\gamma\}$ (which is unique) such that $\alpha+\alpha'=\gamma$.  We denote such an $\alpha'$ by $\theta(\alpha)$. We set $\Gamma_{\gamma}^0=\Gamma_{\gamma}\setminus\{\gamma\}$. 
\end{defi}

We define the formal character ${\rm ch}\,M$ of any $\h$-module $M=\bigoplus_{\nu\in\h^*}M_{\nu}$ with finite-dimensional weight subspaces $M_{\nu}:=\{m\in M;\;\forall h\in\h,\,h.m=\nu(h)m\}$
by $${\rm ch}\,M=\sum_{\nu\in\h^*}\dim M_{\nu}\,e^{\nu}$$ where $e^{\mu}e^{\nu}=e^{\mu+\nu}$ for $\mu,\,\nu\in\h^*$ and
we write, for two such $\h$-modules $M,\,N$, ${\rm ch\,M}\le{\rm ch}\,N$ whenever, for all $\nu\in\h^*$, $\dim M_{\nu}\le\dim N_{\nu}$.

We recall the following Lemma, which is a slight refinement of \cite[Lem. 4.2.1]{F01} (itself inspired by \cite[8.6]{J5} and by \cite[Lem. 6.11]{J6bis}).

\begin{lm}{(\cite[Lem. 4.2.1]{F01})}\label{lmAP}
Assume that there exist disjoint subsets of $\Delta(\pi')$ : $S,\,T,\,\Gamma^0_{\gamma}$, for $\gamma\in S$, such that $\Gamma_{\gamma}=\Gamma^0_{\gamma}\sqcup\{\gamma\}$ is an Heisenberg set with centre $\gamma$. Set $O=\bigsqcup_{\gamma\in S}\Gamma^0_{\gamma}$. Then $\theta:O\longrightarrow O$ is an involution. Set $$y=\sum_{\gamma\in S}x_{-\gamma}\in\widetilde{\p_{\Lambda}}^*.$$
Denote by $\widetilde\Phi_y$ the skew-symmetric bilinear form on $\widetilde{\p_{\Lambda}}\times\widetilde{\p_{\Lambda}}$ defined by $$\widetilde\Phi_y(x,\,x')=K(y,\,[x,\,x']_{\widetilde\p})$$ for all $x,\,x'\in\widetilde{\p_{\Lambda}}$, where $K$ is the Killing form on $\g$. Set also $\g_O=\bigoplus_{\alpha\in O}\mathbb K x_{\alpha},\;\g_{-T}=\bigoplus_{\alpha\in T}\mathbb K x_{-\alpha}$.\smallskip

Assume further that:
\begin{enumerate}
\item[(i)] $S_{\mid \h_{\pi'}\oplus\h_{\Lambda}}$ is a basis for $(\h_{\pi'}\oplus\h_{\Lambda})^*$. \smallskip

\item[(ii)] $\Delta(\pi')=\bigsqcup_{\gamma\in S}\Gamma_{\gamma}\sqcup T.$\smallskip

\item[(iii)] $\lvert T\rvert={\rm index}\,\p_{\Lambda}.$ \smallskip

\item[(iv)]
 The restriction of $\widetilde\Phi_y$ to $\g_O\times\g_O$ is nondegenerate.
\end{enumerate}
Then
\begin{equation}{\rm ad}^*\widetilde{\p_{\Lambda}}(y)\oplus\g_{-T}=\widetilde{\p_{\Lambda}}^*\label{eg}\end{equation}
where ${\rm ad}^*$ denotes the coadjoint action of $\widetilde{\p_{\Lambda}}$ on its dual space. This means that $y$ is regular in $\widetilde{\p_{\Lambda}}^*$ and if  $h\in\h_{\pi'}\oplus\h_{\Lambda}$ is such that $\gamma(h)=1$ for all $\gamma\in S$, then $(h,\,y)$ is an adapted pair for $\widetilde{\p_{\Lambda}}$.

Moreover for all $\gamma\in T$, denote by $t(\gamma)$ the unique element in $\mathbb QS$ such that $\gamma+t(\gamma)$ vanishes on $\h_{\pi'}\oplus\h_{\Lambda}$. 
Assume furthermore that:
\begin{enumerate}
\item[(v)] $\gamma+t(\gamma)\neq 0$ and $t(\gamma)\in\mathbb NS$, for all $\gamma\in T$.
\end{enumerate} Then one has that:
\begin{equation}{\rm ch}\,Y\big(\widetilde{\p_{\Lambda}}\big)\le\prod_{\gamma\in T}\big(1-e^{\gamma+t(\gamma)}\big)^{-1}.\label{upperbound}\end{equation}

Finally assume that:
\begin{enumerate}

\item[(vi)] equality holds in the above inequality.
\end{enumerate} Then restriction of functions gives the algebra isomorphism:
\begin{equation}Y\big(\widetilde{\p_{\Lambda}}\big)\xrightarrow{\sim}\mathbb K[y+\g_{-T}]\label{WS}\end{equation}
which means that $y+\g_{-T}$ is a Weierstrass section for $Y\big(\widetilde{\p_{\Lambda}}\big)$ and $Y\big(\widetilde{\p_{\Lambda}}\big)$ is a polynomial $\mathbb K$-algebra.
\end{lm}

\begin{Rq}
 By Lemma \ref{indexindex2}, Eq. \ref{eg} implies that the element $y\in\widetilde{\p_{\Lambda}}^*$ is regular. Actually by \cite[Proof of Lem. 3.6.1]{F01} or an analogue of it,  we can easily check that Eq. \ref{eg} also implies that $y$ is regular without  Lemma \ref{indexindex2}, since $\lvert T\rvert={\rm index}\,\p_{\Lambda}=\dim\g_{-T}$ and since the index cannot decrease under contraction (moreover a direct sum is not needed in Eq. \ref{eg}, only a sum is needed which is then direct by  subsection \ref{reg}). Thus the result of Lemma \ref{indexindex2}  would also have been be obtained as a consequence of Eq. \ref{eg}. 
\end{Rq}
The construction of the disjoint Heisenberg sets is inspired by the Heisenberg sets constructed in \cite{J1} from the Kostant cascade of a semisimple Lie algebra. We recall in the following subsection some useful results in \cite{J1}.

\subsection{Some results on Heisenberg sets constructed from the Kostant cascade.}\label{Kcasc}

 Let us recall some results in \cite{J1} for $\g$ simple of type ${\rm A}_{n-1}$ with $n\ge 2$ (with a fixed Cartan subalgebra $\h$, a root system $\Delta$ of $(\g,\h)$ and a chosen set of simple roots $\pi$).  To the Kostant cascade $\mathscr K(\g)=\{1,\,2,\,\ldots,\,[n/2]\}$ of $\g$ is attached the set of strongly orthogonal positive roots  $\beta_i=\ep_i-\ep_{n+1-i}$, for $i\in\mathscr K(\g)$, defined inductively as follows. One has $\beta_1=\varpi_1+\varpi_{n-1}$ which is the highest root of $\Delta_1:=\Delta$. Then $\beta_2$ is the highest root of $\Delta_2:=\{\alpha\in\Delta_1\mid (\alpha,\,\beta_1)=0\}$ where $(\,,\,)$ is the bilinear form defined in subsection \ref{notstand}, and so on... 
 With every positive root $\beta_i$, $i\in\mathscr K(\g)$, is associated an Heisenberg set ${\rm H}_{\beta_i}$ with centre $\beta_i$, included in $\Delta^+_i:=\Delta_i\cap\Delta^+$, such that 
 $${\rm H}_{\beta_i}=\{\gamma\in\Delta_i\mid (\gamma,\,\beta_i)>0\}.$$
 
 For example, ${\rm H}_{\beta_1}=\{\beta_1,\,\ep_1-\ep_j,\,\ep_j-\ep_n;\,2\le j\le n-1\}.$
 
 Recall the following very useful Lemma. 
 
 \begin{lm}{(\cite[Lem. 2.2]{J1})}\label{strong}
 We keep the above notation and hypotheses. Then:
 \begin{enumerate}
\item[{\rm (i)}] The set  $\Delta^+$ of positive roots is the disjoint union of the ${\rm H}_{\beta_i}$ for $i\in\mathscr K(\g)$. 
\item[{\rm (ii)}] For every $i,\,j\in\mathscr K(\g)$, with $i\le j$, given $\gamma\in {\rm H}_{\beta_i}$, $\delta\in {\rm H}_{\beta_j}$, then $\gamma+\delta\in\Delta$ implies that $\gamma+\delta\in {\rm H}_{\beta_i}$.
\item[{\rm (iii)}] Set ${\rm H}^0_{\beta_i}={\rm H}_{\beta_i}\setminus\{\beta_i\}$. If $\gamma\in {\rm H}^0_{\beta_i}$ then $\beta_i-\gamma\in {\rm H}^0_{\beta_i}$. \par\noindent
It follows that:
\item[{\rm (iv)}] For every $i,\,j\in\mathscr K(\g)$, with $i\le j$, given $\gamma\in {\rm H}^0_{\beta_i}$, $\delta\in {\rm H}^0_{\beta_j}$ such that there exists $k\in\mathscr K(\g)$ for which $\gamma+\delta=\beta_k$. Then $k=i$ and $\delta=\beta_i-\gamma\in {\rm H}^0_{\beta_i}$.
 \end{enumerate}
 
 \end{lm}

\subsection{The set S}\label{S}

We have to distinguish between the case when $n\le 3s$ and the case when $n>3s$. Recall that we have assumed that $2s<n$ (see subsection \ref{rootsp}) and that we have set $a=s$ and $b=n-2s$ which are respectively the sizes of the extremal and central blocks of the Levi factor $\mathfrak r$ of $\p$.
For any real number $x$, denote by $[x]$ the unique integer such that
$$[x]\le x<[x]+1.$$

Observe that, if $x\in\mathbb Z$, then \begin{equation}-\bigg[\frac{x}{2}\bigg]=\bigg[\frac{1-x}{2}\bigg].\label{partentier}\end{equation}

Recall that we have set $c=\min(a,\,b)$.
Set \begin{align*}\beta_i=\ep_i-\ep_{n+1-i},\;\forall 1\le i\le s,\;\gamma_{s-j}=\ep_{s-j}-\ep_{n-s-j},\,\forall 0\le j\le c-1\\
\;\delta_k=\ep_{n-s-k}-\ep_{n-s+k},\;\forall 1\le k\le c.\end{align*}

The roots $\beta_i$, for $1\le i\le s$, are the  first $s$ strongly orthogonal positive roots corresponding to  the Kostant cascade of $\g$ (see subsection \ref{Kcasc}).

If $2s<n\le 3s\iff a\ge b$, we set
\begin{align*}\beta'_k=\ep_{2n-3s-1+k}-\ep_{n-k+1},\;\forall 1\le k\le [(3s-n+1)/2]\\
\;\beta''_k=\ep_k-\ep_{3s-n+1-k},\;\forall 1\le k\le [(3s-n)/2].\end{align*}

The roots $\beta'_k$, for $1\le k\le [(3s-n+1)/2]$, correspond to the Kostant cascade of a subalgebra of $\g$ of type ${\rm A}_{3s-n}$ and the roots $\beta''_k$, for $1\le k\le[(3s-n)/2]$, correspond to the Kostant cascade of a subalgebra of $\g$ of type ${\rm A}_{3s-n-1}$ (see explanations in subsection \ref{ST} below for more details).

If $n>3s\iff a<b$, we set
\begin{align*} \beta'_k=\ep_{s+k}-\ep_{n-2s-k+1},\;\forall 1\le k\le [(n-3s)/2]\\
\beta''_k=\ep_{s+k}-\ep_{n-2s-k},\;\forall 1\le k\le [(n-3s-1)/2].\end{align*}

The roots $\beta'_k$, for $1\le k\le [(n-3s)/2]$, correspond to the Kostant cascade of a subalgebra of $\g$ of type ${\rm A}_{n-3s-1}$ and the roots $\beta''_k$, for $1\le k\le[(n-3s-1)/2]$, correspond to the Kostant cascade of a subalgebra of $\g$ of type ${\rm A}_{n-3s-2}$ (see explanations in subsection \ref{ST} below for more details).\smallskip

For the set $S$ we set :
\begin{enumerate}
\item if $2s<n\le 3s$, \begin{align*}S=&\{\beta_i;\;1\le i\le s-1,\;\gamma_{s-j};\;\;0\le j\le b-1,\,\delta_k;\;1\le k\le b-1,\\
&-\beta'_{\ell};\;1\le \ell\le [(3s-n+1)/2],-\beta''_{t};\;1\le t\le [(3s-n)/2]\}\end{align*}
\item
 if $n>3s$,
\begin{align*}S=&\{\beta_i;\;1\le i\le s-1,\;\gamma_{s-j};\;\;0\le j\le s-1,\,\delta_k;\;1\le k\le s,\\
&\beta'_{\ell};\;1\le \ell\le  [(n-3s)/2],-\beta''_{t};\;1\le t\le [(n-3s-1)/2]\}.\end{align*}

\end{enumerate}

\subsection{Our choice of  Heisenberg  sets.}\label{ST}
Let us give now our choice of  Heisenberg sets.\smallskip

Recall that the Levi factor $\mathfrak r$ of $\p$ is isomorphic to $(\g\mathfrak l_a\times\g\mathfrak l_b\times\g\mathfrak l_a)\cap\mathfrak s\mathfrak l_n$, with $a=s$ and $b=n-2s$. Recall also the notation in the previous subsections. Since we want the restriction to $\g_O\times\g_O$ of the bilinear form $\widetilde\Phi_y$ to be nondegenerate, we have to choose the roots in the Heisenberg sets very carefully. Indeed  the Heisenberg sets are required to verify the following {\bf condition (C)} (see \cite[4.3]{F01}).

\begin{equation*}\forall\,\alpha\in O, \lvert\{\alpha,\,\theta(\alpha)\}\cap(\Delta^+\setminus\Delta^+_{\pi'})\rvert\le 1\tag {{\rm\bf C}}\end{equation*}

so that for each root $\alpha\in O$, the set \begin{equation}S_{\alpha}=\{\beta\in O\mid \alpha+\beta\in S\;{\rm and}\;[x_{\alpha},\,x_{\beta}]_{\widetilde\p}\neq 0\}\label{Salph}\end{equation}  at least contains $\theta(\alpha)$. Actually, if for each $\alpha\in O$, $S_{\alpha}=\{\theta(\alpha)\}$, then the restriction to $\g_O\times\g_O$ of the bilinear form $\widetilde\Phi_y$ is nondegenerate, since in this case we obtain, up to a nonzero scalar, that
$$\det({\widetilde\Phi}_{y\mid\g_O\times\g_O})=\prod_{\alpha\in O}K(y,\,[x_{\alpha},\,x_{\theta(\alpha)}]_{\widetilde\p})\neq 0.$$

 Unfortunately this last condition is rarely  fulfilled. (We will see in subsection \ref{nondeg} which sufficient conditions give the required nondegeneracy).

Since $\pi'=\pi\setminus\{\alpha_s,\,\alpha_{n-s}\}$ with $s<n-s$, we can write $\pi'=\pi'_1\sqcup\pi'_2\sqcup\pi'_3$ with $\pi'_1=\{\alpha_1,\,\ldots,\,\alpha_{s-1}\}$, $\pi'_2=\{\alpha_{s+1},\,\ldots,\,\alpha_{n-s-1}\}$ and $\pi'_3=\{\alpha_{n-s+1},\,\ldots,\,\alpha_{n-1}\}$. The root subset of the root system $\Delta$ spanned by $\pi'_i$ ($1\le i\le 3$) will be denoted by $\Delta_{\pi'_i}$ (that is, $\Delta_{\pi'_i}=\Delta\cap\mathbb Z\pi'_i$) and we set $\Delta_{\pi'_i}^\pm=\Delta_{\pi'_i}\cap\Delta^\pm$.
We denote by $\mathfrak r'_i$ the simple Lie subalgebra of the Levi subalgebra $\mathfrak r'$ of $\p$ spanned by root vectors of weight in $\pi'_i$, so that we have that
$$\mathfrak r'=\mathfrak r'_1\times\mathfrak r'_2\times\mathfrak r'_3.$$

We have that $\mathfrak r'_1\simeq\mathfrak r'_3\simeq\mathfrak s\mathfrak l_a$ and that $\mathfrak r'_2\simeq\mathfrak s\mathfrak l_b$.

The  Levi factor $\mathfrak r$ of $\p$ can be represented by three symmetric (with respect to the antidiagonal) blocks, which we denote by $B_i$ : each $B_i$ (except its diagonal $D_i$) corresponds to $\Delta_{\pi'_i}$ and $B_i^\pm$ corresponds to $\Delta_{\pi'_i}^\pm$.

Observe that the roots in $\Delta^+\setminus\Delta^+_{\pi'}$ (which are the weights of the root vectors in the nilpotent radical $\m$ of $\p$) fill two rectangles $R_1$ and $R_2$ and one square $SQ$. The rectangle $R_1$ is  of height $a$ and width $b$ and lies on the right of the upper left  block $B_1$ and above the central block $B_2$. The rectangle $R_2$ is  of height $b$ and of width $a$ and lies on the right of the central block $B_2$ and above the bottom right block $B_3$. The square $SQ$ has its side of length $a$ and lies on the top right corner.
More precisely we have :
\begin{align*} SQ=\{\ep_i-\ep_j;\;1\le i\le s,\;n-s+1\le j\le n\},\\
\;R_1=\{\ep_i-\ep_j;\;1\le i\le s,\,s+1\le j\le n-s\},\\
\;R_2=\{\ep_i-\ep_j;\;s+1\le i\le n-s,\,n-s+1\le j\le n\}.\end{align*}

Observe that a root $\gamma\in SQ$ is of the form
\begin{equation}\gamma=\cdots+\alpha_s+\cdots+\alpha_{n-s}+\cdots\label{expSQ}\end{equation}

A root $\gamma\in R_1$ is of the form
\begin{equation}\gamma=\cdots+\alpha_s+\cdots\label{expR1}\end{equation}

A root $\gamma\in R_2$ is of the form
\begin{equation}\gamma=\cdots+\alpha_{n-s}+\cdots\label{expR2}\end{equation}

where  there is possibly a sum of simple roots belonging to $\pi'$ in the dotted lines.

In Figures \ref{Fig1} and \ref{Fig2}, we have drawn with vertical red lines the rectangles $R_1$, $R_2$ and square $SQ$, which represent the nilpotent radical $\m$ of $\p$ and the three blocks $B_i=B_i^-\sqcup B_i^+\sqcup D_i$, $i\in\{1,\,2,\,3\}$, drawn with north east blue lines, which represent the Levi factor $\mathfrak r$ of $\p$. We also have drawn with north west  lines the sets $B_{1,\,r}^-\subset B_1^-$ and $B_{3,\,r}^-\subset B_3^-$ when $a\ge b$ or the sets $B_{2,\,r}^-\subset B_2^-$ and $B_{2,\,r}^+\subset B_2^+$ when $a<b$ of socalled {\it remaining roots} as defined in the description of Heisenberg sets below. Finally we also have indicated the different elements of the set $S\cap(\Delta^+\setminus\Delta^+_{\pi'})$ on the Figures. The rest of elements in $S\cap\Delta_{\pi'}$ are the elements on the antidiagonal of each set of remaining roots, corresponding to (opposite) of the Kostant cascade of each simple Lie algebra spanned by root vectors whose weight belongs to each of these sets.
Recall that we have $2a+b=n$.

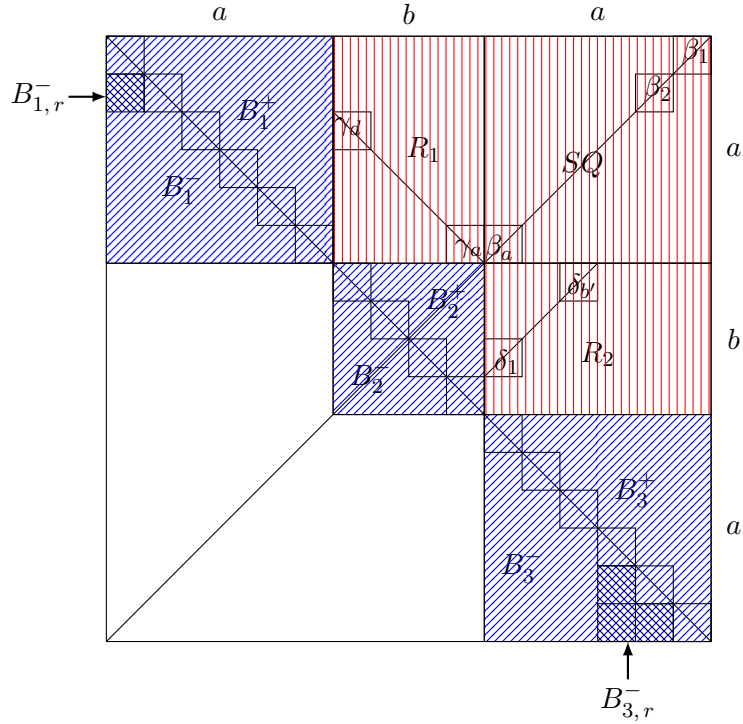
\begin{figure}[b]
\caption{The case $a\ge b$.}\label{Fig1} 

\medskip
\begin{center}
\begin{tikzpicture}
\node[anchor=center] at (6.5,3.8) {$R_2$};
\node[anchor=center] at (4.2,6.5) {$R_1$};
\node[anchor=center] at (6.3,6.3) {$SQ$};
\node[anchor=center] at (1,6) {$B_1^-$};
\draw[->, >=latex, thick] (-0.5,7.2)--(0,7.2);
\node at (-0.9,7.2) {$B_{1,\,r}^-$};
\node[anchor=center] at (2,7) {$B_1^+$};
\node[anchor=center] at (5.5,1) {$B_3^-$};
\draw[->, >=latex, thick] (6.9,-0.5)--(6.9,0);
\node at (6.9,-0.8) {$B_{3,\,r}^-$};
\node[anchor=center] at (7,2) {$B_3^+$};
\node[anchor=center] at (4.5,4.5) {$B_2^+$};
\node[anchor=center] at (3.5,3.5) {$B_2^-$};
\node[anchor=center] at (1.5,8.3) {$a$};
\node[anchor=center] at (8.3,6.5) {$a$};
\node[anchor=center] at (4,8.3) {$b$};
\node[anchor=center] at (8.3,4) {$b$};
\node[anchor=center] at (6.5,8.3) {$a$};
\node[anchor=center] at (8.3,1.5) {$a$};
\node[anchor=center] at (7.8,7.8) {$\beta_1$};
\draw (7.5,7.5) rectangle (8,8);
\draw (7,7) rectangle (7.5,7.5);
\node[anchor=center] at (7.3,7.3) {$\beta_2$};
\draw (5,5) rectangle (5.5,5.5);
\node[anchor=center] at (5.2,5.2) {$\beta_a$};
\draw (4.5,5.5) rectangle (5,5);
\node[anchor=center] at (4.8,5.2) {$\gamma_a$};
\draw (3,7) rectangle (3.5,6.5);
\node[anchor=center] at (3.2,6.8) {$\gamma_{d}$};
\draw (5,5)--(3,7);
\draw (5,3.5)--(6.5,5);
\draw (5,3.5) rectangle (5.5,4);
\node[anchor=center] at (5.3,3.7) {$\delta_1$};
\draw (6,4.5) rectangle (6.5,5);
\node[anchor=center] at (6.3,4.7) {$\delta_{b'}$};
\draw (0,0) -- (8,0) ;
\draw (0,0) -- (0,8) ;
\draw (8,0) -- (8,8) ;
\draw (0,8) -- (8,8) ;
\draw (3,3) -- (3,8);
\draw (3,3) -- (8,3);
\draw (0,5) -- (3,5);
\draw (5,0)-- (5,3);
\draw (3,3) -- (5,3);
\draw (5,5) -- (5,8);
\draw (5,5)--(8,5);
\draw (0,0) -- (8,8);
\draw (0,8) -- (8,0);
\draw[pattern=north west lines] (6.5,0) rectangle (7,1);
\draw[pattern=north west lines] (7,0) rectangle (7.5,0.5);
\draw (6.5,0.5)--(7,0.5);
\draw[pattern=north west lines] (0,7) rectangle (0.5,7.5);
\draw[pattern=north east lines, pattern color=blue] (0,5) rectangle (3,8);
\draw[pattern=north east lines, pattern color=blue] (3,3) rectangle (5,5);
\draw[pattern=north east lines, pattern color=blue] (5,0) rectangle (8,3);
\draw[pattern=vertical lines, pattern color=red] (3,5) rectangle (5,8);
\draw[pattern=vertical lines, pattern color=red] (5,3) rectangle (8,5);
\draw[pattern=vertical lines, pattern color=red] (5,5) rectangle (8,8);
\draw (0,8) rectangle (0.5,7.5);
\draw (0.5,7.5) rectangle (1,7);
\draw (1,7) rectangle (1.5,6.5);
\draw (1.5,6.5) rectangle (2,6);
\draw (2,6) rectangle (2.5,5.5);
\draw (2.5,5.5) rectangle (3,5);
\draw (3,5) rectangle (3.5,4.5);
\draw (3.5,4.5) rectangle (4,4);
\draw (4,4) rectangle (4.5,3.5);
\draw (4.5,3.5) rectangle (5,3);
\draw (5,3) rectangle (5.5,2.5);
\draw (5.5,2.5) rectangle (6,2);
\draw (6,2) rectangle (6.5,1.5);
\draw (6.5,1.5) rectangle (7,1);
\draw (7,1) rectangle (7.5,0.5);
\draw (7.5,0.5) rectangle (8,0);
\end{tikzpicture}\medskip

\hbox{\rm where we set $d=a-b+1=3s-n+1$} \par
\centering{\hbox{\rm and $b'=b-1=n-2s-1$ if $b-1\ge 1$.}}
\end{center}
\end{figure}

\medskip

\begin{figure}[b]
\caption{The case $a<b$.}
\label{Fig2}
\begin{center}
\begin{tikzpicture}
\draw (0,0)--(8,8);
\draw (0,8)--(8,0);
\draw (0,0) rectangle (8,8);
\draw (0,8) rectangle (2,6);
\draw (2,6) rectangle (6,2);
\draw (6,2) rectangle (8,0);
\draw (6,6)--(6,8);
\draw (6,6)--(8,6);
\draw (6,6)--(4,8);
\draw (6,2.5)--(8,4.5);
\draw (7.5,7.5) rectangle (8,8);
\node[anchor=center]  at (7.75,7.75) {$\beta_1$};
\draw (6,6) rectangle (6.5,6.5);
\node at (6.25,6.25) {$\beta_s$};
\draw (6,6) rectangle (5.5,6.5);
\node at (5.75,6.25) {$\gamma_s$};
\draw (4,8) rectangle (4.5,7.5);
\node at (4.25,7.75) {$\gamma_1$};
\draw (6,2.5) rectangle (6.5,3);
\node at (6.25,2.75) {$\delta_1$};
\draw (7.5,4) rectangle (8,4.5);
\node at (7.75,4.25) {$\delta_s$};
\draw[pattern=north east lines, pattern color=blue] (0,8) rectangle (2,6);
\draw[pattern=north east lines, pattern color=blue] (2,6) rectangle (6,2);
\draw[pattern=north east lines, pattern color=blue] (6,2) rectangle (8,0);
\draw[pattern=vertical lines, pattern color=red] (6,6) rectangle (2,8);
\draw[pattern=vertical lines, pattern color=red] (6,6) rectangle (8,8);
\draw[pattern=vertical lines, pattern color=red] (6,2) rectangle (8,6);
\node at (1.5,7.5) {$B_1^+$};
\node at (0.5,6.5) {$B_1^-$};
\node at (5,5) {$B_2^+$};
\draw[->, >=latex, thick] (1.5,5)--(2,5);
\draw[->, >=latex, thick] (3.25,6.5)--(3.25,6);
\node at (3.25,6.8) {$B_{2,\,r}^+$};
\node at (3,3) {$B_2^-$};
\node at (1,5) {$B_{2,\,r}^-$};
\node at (7.5,1.5) {$B_3^+$};
\node at (6.5,0.5) {$B_3^-$};
\node at (4,7) {$R_1$};
\node at (7,4) {$R_2$};
\node at (7,7) {$SQ$};
\node at (1,8.3) {$a$};
\node at (4,8.3) {$b$};
\node at (7,8.3) {$a$};
\node at (8.3,1) {$a$};
\node at (8.3,4) {$b$};
\node at (8.3,7) {$a$};
\draw[pattern=north west lines] (2,4.5) rectangle (2.5,5);
\draw[pattern=north west lines] (2,5) rectangle (2.5,5.5);
\draw[pattern=north west lines] (2.5,4.5) rectangle (3,5);
\draw[pattern=north west lines] (3.5, 5) rectangle (4,5.5);
\draw[pattern=north west lines] (3.5, 4.5) rectangle (4,5);
\draw[pattern=north west lines] (3.5, 5.5) rectangle (4,6);
\draw[pattern=north west lines] (3, 5.5) rectangle (3.5,6);
\draw[pattern=north west lines] (2.5, 5.5) rectangle (3,6);
\draw[pattern=north west lines] (3, 5) rectangle (3.5,5.5);
\draw (0,8) rectangle (0.5,7.5);
\draw (0.5,7.5) rectangle (1,7);
\draw (1,7) rectangle (1.5,6.5);
\draw (1.5,6.5) rectangle (2,6);
\draw (2,6) rectangle (2.5,5.5);
\draw (2.5,5.5) rectangle (3,5);
\draw (3,5) rectangle (3.5,4.5);
\draw (3.5,4.5) rectangle (4,4);
\draw (4,4) rectangle (4.5,3.5);
\draw (4.5,3.5) rectangle (5,3);
\draw (5,3) rectangle (5.5,2.5);
\draw (5.5,2.5) rectangle (6,2);
\draw (6,2) rectangle (6.5,1.5);
\draw (6.5,1.5) rectangle (7,1);
\draw (7,1) rectangle (7.5,0.5);
\draw (7.5,0.5) rectangle (8,0);

\end{tikzpicture}
\end{center}
\end{figure}
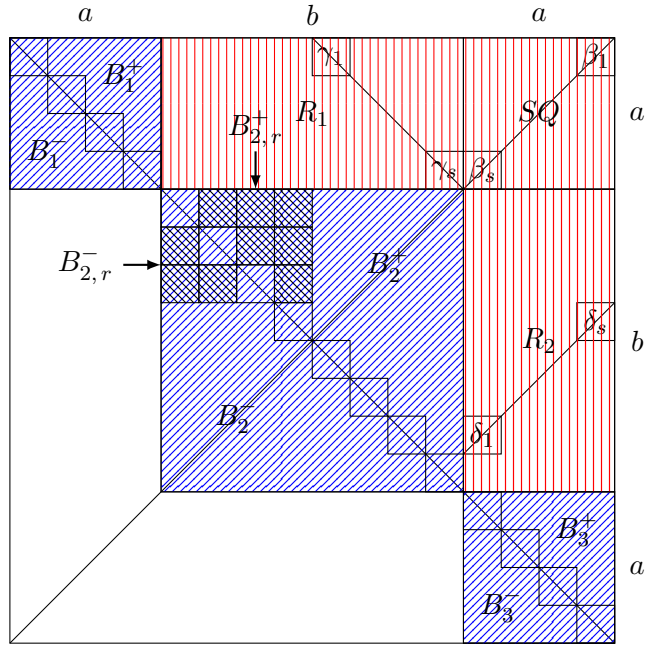

The first $s$ strongly orthogonal positive roots $\beta_i=\ep_i-\ep_{n+1-i}$ corresponding to the Kostant cascade of $\g$ are the weights of the first $s$ matrices on the antidiagonal, all lying in the square $SQ$, the top right one corresponding to $\beta_1=\ep_1-\ep_n$ (see Fig. \ref{Fig1} and Fig. \ref{Fig2}). For $1\le i\le s-1$,  the roots $\alpha\in(\Delta^+\setminus\Delta^+_{\pi'})\cap\Gamma_{\beta_i}^0$ are the roots lying in the  square $SQ$ on the same column as $\beta_i$ and below $\beta_i$ or the roots in $SQ$ on the same row as $\beta_i$ and on the left of $\beta_i$, so that the roots $\theta(\alpha)$  belong to $\Delta^+_{\pi'}$, more precisely they belong to $B_1^+$ or to $B_3^+$.
In other words we set, for every $1\le i\le s-1$,
\begin{align*}\Gamma_{\beta_i}=\{\beta_i,\,\ep_i-\ep_j,\,\ep_j-\ep_{n+1-i};\;n-s+1\le j\le n-i,\\
\;\ep_k-\ep_{n+1-i},\,\ep_i-\ep_k,\;i+1\le k\le s\}.\end{align*}
By construction, we have that $\Gamma_{\beta_i}$ verifies condition ({\bf C}) and that $\Gamma_{\beta_i}\subset\Delta^+$. We also remark that $\Gamma_{\beta_i}\subset {\rm H}_{\beta_i}$ with the notation of subsection \ref{Kcasc}.

Moreover the square $SQ$ (except $\beta_s$) together with $B_1^+$ and $B_3^+$ are entirely covered  by the union of the sets $\Gamma_{\beta_i}$, for $1\le i\le s-1$. In other words we have that
\begin{equation}\bigsqcup_{1\le i\le s-1}\Gamma_{\beta_i}\sqcup\{\beta_s\}=\{\ep_i-\ep_j;\;1\le i\le s,\,n-s+1\le j\le n\}\sqcup\Delta^+_{\pi'_1}\sqcup\Delta^+_{\pi'_3}.\label{decompSQ}\end{equation}

For all $0\le j\le c-1$ ($c=\min(a,\,b)$), the roots $\gamma_{s-j}=\ep_{s-j}-\ep_{n-s-j}$ lie  in the rectangle $R_1$ on a line which is parallel to the diagonal with $\gamma_s$ at the bottom right of $R_1$ (see Fig. \ref{Fig1} and Fig. \ref{Fig2}). For $0\le j\le c-1$, the roots $\alpha\in(\Delta^+\setminus\Delta^+_{\pi'})\cap\Gamma_{\gamma_{s-j}}^0$  are the roots  lying in the rectangle $R_1$ on the same column as $\gamma_{s-j}$ and above it or on the same row as $\gamma_{s-j}$ and on its left, so that the roots $\theta(\alpha)\in\Delta_{\pi'}$. More precisely the roots $\theta(\alpha)$ for $\alpha$ lying on the left of $\gamma_{s-j}$ lie in $B_2^+$ and the roots $\theta(\alpha)$ for $\alpha$ lying above $\gamma_{s-j}$ lie in $B_1^-$.

In other words we set, for every $0\le j\le c-1$,
\begin{align*}\Gamma_{\gamma_{s-j}}=\{\gamma_{s-j},\,\ep_k-\ep_{n-s-j},\,\ep_{s-j}-\ep_k,\;1\le k\le s-j-1,\\
\;\ep_{s-j}-\ep_{\ell},\,\ep_{\ell}-\ep_{n-s-j},\;s+1\le \ell\le n-s-j-1\}.\end{align*}
By construction, we have that $\Gamma_{\gamma_{s-j}}$ verifies condition ({\bf C}) and the roots of $\Gamma_{\gamma_{s-j}}$ lie both in $\Delta^+$ and in $\Delta^-_{\pi'}$.

For all $1\le k\le c-1$ (if $a\ge b$) or for all $1\le k\le c$ (if $a<b$), the roots $\delta_k=\ep_{n-s-k}-\ep_{n-s+k}$ lie in the rectangle $R_2$ on a line which is parallel to the antidiagonal, the leftmost one, corresponding to $\delta_1$, lying on the second row of $R_2$ if we count from bottom to top (see Fig. \ref{Fig1} and Fig. \ref{Fig2}). For $1\le k\le c-1$ (if $a\ge b$) or for $1\le k\le c$ (if $a<b$), the roots $\alpha\in(\Delta^+\setminus\Delta^+_{\pi'})\cap\Gamma_{\delta_k}^0$ are the roots lying in the rectangle $R_2$ on the same column as $\delta_k$ and above it or on the same row as $\delta_k$ and on its right, so that the roots $\theta(\alpha)\in\Delta_{\pi'}$. More precisely the roots $\theta(\alpha)$ for $\alpha$ lying on the right of $\delta_k$ lie in 
$B_3^-$ and the roots $\theta(\alpha)$ for $\alpha$ lying above $\delta_k$ lie in $B_2^-$.

In other words we set, for every $1\le k\le c-1$, if $a\ge b$, or for every $1\le k\le c$, if $a<b$ :
\begin{align*}\Gamma_{\delta_k}=\{\delta_k,\,\ep_{n-s-k}-\ep_{\ell},\;\ep_{\ell}-\ep_{n-s+k};\;n-s+k+1\le \ell\le n,\\
\ep_{\ell'}-\ep_{n-s+k},\,\ep_{n-s-k}-\ep_{\ell'};\;s+1\le\ell'\le n-s-k-1\}.
 \end{align*}
 By construction, we have that $\Gamma_{\delta_k}$ verifies condition ({\bf C}) and the roots of $\Gamma_{\delta_k}$ lie both in $\Delta^+$ and in $\Delta^-_{\pi'}$.
 
Using the above process, we obtain by construction  disjoint Heisenberg sets which cover the rectangles $R_1$, $R_2$ and the square $SQ$, except for the root $\beta_s=\ep_s-\ep_{n-s+1}$ and the roots $\ep_{n-s}-\ep_k$ for $n-s+1\le k\le n$ (which correspond to the bottom row of the rectangle $R_2$) and part of the blocks of the Levi factor.
 
 One can check that the above process implies that the set $\{\ep_{n-s}-\ep_k;\;s+1\le k\le n-s-1\}\subset\Delta^-_{\pi'}$ (which corresponds to the bottom row of $B_2^-$) has an empty intersection with the union of the Heisenberg sets constructed above.
 
 The roots in $\Delta_{\pi'}\setminus\{\ep_{n-s}-\ep_k;\;s+1\le k\le n-s-1\}$ which do not belong to the union of the above Heisenberg sets will be called {\it remaining roots}. We will precise below how we can construct other Heisenberg sets which cover these remaining roots.
 
 Assume  that $a\ge b$. Then one can check that $B_2$, except its bottom row,  is entirely covered by the already built Heisenberg sets. It follows that the remaining roots lie only in $B_1^-$ or in $B_3^-$. More precisely one can check that the remaining roots in $B_1^-$ lie in the top left  corner $B_{1,\,r}^-\subset B_1^-$ corresponding  to the negative part of the simple Lie subalgebra $\mathfrak r'_{1,\,r}\subset\mathfrak r'_1$ of type ${\rm A}_{3s-n-1}$ spanned by root vectors whose weight is equal to one of the  first $3s-n-1$ simple roots of $\g$. We consider the strongly orthogonal positive roots $\beta''_k$, for $1\le k\le[(3s-n)/2]$, corresponding to the Kostant cascade of  $\mathfrak r'_{1,\,r}$ and  we take in $S$ the $-\beta''_k$ , with $-\beta''_1$ being the bottom left root of $B_{1,\,r}^-$ : the roots $-\beta''_k$ lie on the antidiagonal of the set $B_{1,\,r}^-$.
 In other words we have:
 $$\beta''_k=\ep_k-\ep_{3s-n+1-k},\;\forall 1\le k\le [(3s-n)/2].$$
 Then the Heisenberg set $\Gamma_{-\beta''_k}$ consists in roots in $B_1^-$ (and even in $B_{1,\,r}^-$) lying on the same column and above $-\beta''_k$ or on the same row and on the right of $-\beta''_k$. 
 In other words we have: 
 $\forall\;1\le k\le [(3s-n)/2]$,
 \begin{align*}\Gamma_{-\beta''_k}=\{-\beta''_k,\; \ep_{3s-n+1-k}-\ep_{\ell},\;\ep_{\ell}-\ep_k;\\
k+1 \le\ell\le 3s-n-k\} \subset\Delta^-_{\pi'_1}.    \end{align*}

 Denoting by ${\rm H}''_{\beta''_k}$ the Heisenberg set with centre $\beta''_k$ of the simple Lie algebra $\mathfrak r'_{1,\,r}$ as defined in subsection \ref{Kcasc} and by $-{\rm H}''_{\beta''_k}$ the set formed by the opposite of roots in ${\rm H}''_{\beta''_k}$, we have that $\Gamma_{-\beta''_k}=-{\rm H}''_{\beta''_k}$.
 By Lemma \ref{strong} (i) the set $B_{1,\,r}^-$ of remaining roots in $B_1^-$ is such that 
 \begin{equation}B_{1,\,r}^-=\bigsqcup_{1\le k\le[(3s-n)/2]}\Gamma_{-\beta''_k}.\label{remfirst}\end{equation}
 
 Similarly, one can check that the remaining roots in $B_3^-$ lie in the bottom right corner $B_{3,\,r}^-\subset B_3^-$ corresponding to the negative part  of the simple Lie subalgebra $\mathfrak r'_{3,\,r}\subset\mathfrak r'_3$ of type ${\rm A}_{3s-n}$ spanned by root vectors whose weight is equal to one of the last $3s-n$ simple roots of $\g$.  Considering the strongly orthogonal positive roots $\beta'_k$, for $1\le k\le[(3s-n+1)/2]$, corresponding to the Kostant cascade of $\mathfrak r'_{3,\,r}$, we take in $S$ the roots $-\beta'_k$ for $1\le k\le[(3s-n+1)/2]$, with $-\beta'_1$ being the bottom left root of $B_{3,\,r}^-$ : the roots $-\beta'_k$ lie on the antidiagonal of the set $B_{3,\,r}^-$. In other words we have:
 $$\beta'_k=\ep_{2n-3s-1+k}-\ep_{n-k+1},\;\forall 1\le k\le [(3s-n+1)/2].$$
 
 Then the Heisenberg set $\Gamma_{-\beta'_k}$ consists in roots in $B_3^-$ (and even in $B_{3,\,r}^-$) lying on the same column and above $-\beta'_k$ or on the same row and on the right of $-\beta'_k$. In other words we have:
 $\forall\;1\le k\le [(3s-n+1)/2]$,
 \begin{align*}\Gamma_{-\beta'_k}=\{-\beta'_k,\;  \ep_{n-k+1}-\ep_{\ell},\ep_{\ell}-\ep_{2n-3s-1+k};\\
 \;2n-3s+k\le\ell\le n-k\} \subset\Delta^-_{\pi'_3}.       \end{align*}
 
  Denoting by ${\rm H}'_{\beta'_k}$ the Heisenberg set with centre $\beta'_k$ of the simple Lie algebra $\mathfrak r'_{3,\,r}$ as defined in subsection \ref{Kcasc} and by $-{\rm H}'_{\beta'_k}$ the set formed by the opposite of roots in ${\rm H}'_{\beta'_k}$, we have that $\Gamma_{-\beta'_k}=-{\rm H}'_{\beta'_k}$.
 By Lemma \ref{strong} (i)  the set $B_{3,\,r}^-$ of remaining roots in $B_3^-$ is such that
 \begin{equation}B_{3,\,r}^-=\bigsqcup_{1\le k\le[(3s-n+1)/2]}\Gamma_{-\beta'_k}.\label{remsecond}\end{equation}
 The sets $B_{1,\,r}^-$ and $B_{3,\,r}^-$ of remaining roots are represented with north west lines on Figure \ref{Fig1}.
 
 Finally assume that $a<b$. Then one can check that $B_1^-$ and $B_3^-$ are entirely covered by the already built Heisenberg sets. In $B_2^-$, one can check that the remaining roots lie in the top left corner $B_{2,\,r}^-\subset B_2^-$ corresponding to the negative part of the simple Lie subalgebra $\mathfrak r'_{2,\,r^-}\subset\mathfrak r'_2$ of type  ${\rm A}_{n-3s-2}$ spanned by root vectors whose weight belongs to the subset $\{\alpha_{s+1},\,\ldots,\,\alpha_{n-2s-2}\}$ of simple roots of $\g$. We take in $S$ the  opposite of the strongly orthogonal positive roots $\beta''_k$, for $1\le k\le[(n-3s-1)/2]$, corresponding to the Kostant cascade of $\mathfrak r'_{2,r^-}$, with $-\beta''_1$ being the bottom left root in $B_{2,\,r}^-$ : the roots $-\beta''_k$ lie on the antidiagonal of $B_{2,\,r}^-$. In other words, we have:
 $$\beta''_k=\ep_{s+k}-\ep_{n-2s-k},\;\forall 1\le k\le [(n-3s-1)/2].$$
 
 Then the Heisenberg set $\Gamma_{-\beta''_k}$ consists in roots in $B_2^-$ (and even in $B_{2,\,r}^-$) lying on the same column and above $-\beta''_k$ or on the same row and on the right of $-\beta''_k$. In other words we have:
 $\forall\;1\le k\le[(n-3s-1)/2]$,
 \begin{align*}\Gamma_{-\beta''_k}=\{-\beta''_k,\; \ep_{n-2s-k}-\ep_{\ell},\,\ep_{\ell}-\ep_{s+k};\\
  s+k+1\le\ell\le n-2s-k-1\}\subset\Delta^-_{\pi'_2}. \end{align*}
 Denoting by ${\rm H}''_{\beta''_k}$ the Heisenberg set with centre $\beta''_k$ of the simple Lie algebra $\mathfrak r'_{2,\,r^-}$ as defined in subsection \ref{Kcasc} and by $-{\rm H}''_{\beta''_k}$ the set formed by the opposite of roots in ${\rm H}''_{\beta''_k}$, we have that $\Gamma_{-\beta''_k}=-{\rm H}''_{\beta''_k}$.
 
 Again by Lemma \ref{strong} (i), the set $B_{2,\,r}^-$ of remaining roots in $B_2^-$ is such that
 \begin{equation}B_{2,\,r}^-=\bigsqcup_{1\le k\le[(n-3s-1)/2]}\Gamma_{-\beta''_k}.\label{remthird}\end{equation}
 
  In $B_2^+$ one can check that the remaining roots lie in the top left corner $B_{2,\,r}^+\subset B_2^+$ corresponding to the positive part of the simple Lie subalgebra $\mathfrak r'_{2,\,r}\subset\mathfrak r'_2$ of type   ${\rm A}_{n-3s-1}$, spanned by root vectors whose weight belongs to the subset $\{\alpha_{s+1},\ldots,\,\alpha_{n-2s-1}\}$ of simple roots of $\g$. We take in $S$ the strongly orthogonal positive roots $\beta'_k$, for $1\le k\le[(n-3s)/2]$, corresponding to the Kostant cascade of $\mathfrak r'_{2,r}$, with $\beta'_1$ being the top right root in $B_{2,\,r}^+$ : the roots $\beta'_k$ lie on the antidiagonal of $B_{2,\,r}^+$. (Observe that we have $\mathfrak r'_{2,\,r^-}\subset\mathfrak r'_{2,\,r}$). In other words, we have:
 $$\beta'_k=\ep_{s+k}-\ep_{n-2s-k+1},\;\forall 1\le k\le [(n-3s)/2].$$
  
  Then the Heisenberg set $\Gamma_{\beta'_k}$ consists in roots in $B_2^+$ (and even in $B_{2,\,r}^+$) lying on the same column and below $\beta'_k$ or on the same row and on the left of $\beta'_k$. In other words, we have:
  $\forall\;1\le k\le[(n-3s)/2]$,
 \begin{align*}\Gamma_{\beta'_k}=\{\beta'_k,\;\ep_{s+k}-\ep_{\ell},\;\ep_{\ell}-\ep_{n-2s-k+1};\\
        s+k+1\le\ell\le n-2s-k\}\subset\Delta^+_{\pi'_2}.       \end{align*}
  
  Denoting by ${\rm H}'_{\beta'_k}$ the Heisenberg set with centre $\beta'_k$ of the simple Lie algebra $\mathfrak r'_{2,\,r}$ as defined in subsection \ref{Kcasc}, we have that $\Gamma_{\beta'_k}={\rm H}'_{\beta'_k}$.

  By Lemma \ref{strong} (i), the set  $B_{2,\,r}^+$ of remaining roots in $B_2^+$ is such that
  \begin{equation}B_{2,\,r}^+ =\bigsqcup_{1\le k\le[(n-3s)/2]}\Gamma_{\beta'_k}.\label{remfourth}\end{equation}
The sets $B_{2,\,r}^-$ and $B_{2,\,r}^+$ of remaining roots are represented with north west lines on Figure \ref{Fig2}.
 
By construction, in both cases ($a\ge b$ or $a<b$), all the Heisenberg sets $\Gamma_{\gamma}$, for $\gamma\in S$,  are disjoint and verify condition $({\bf C})$.

\subsection{The set T}\label{T}

By construction in subsection \ref{ST} of our Heisenberg sets, we can see that, in both cases,  the set $T=(\Delta^+\sqcup\Delta^-_{\pi'})\setminus\bigsqcup_{\gamma\in S}\Gamma_{\gamma}$ is the following.
\begin{align*}T=\{\ep_s-\ep_{n-s+1},\;\ep_{n-s}-\ep_k;\;s+1\le k\le n,\,k\neq n-s\}.\end{align*} 
We check that $\lvert T\rvert=n-s={\rm index}\,\p_{\Lambda}={\rm index}\,\widetilde{\p_{\Lambda}}$ by Lemmas \ref{indexp} and \ref{indexindex2}.
  Summarizing we have obtained disjoint Heisenberg sets $\Gamma_{\gamma}$, $\gamma\in S$, verifying condition $({\bf C})$ and such that $$\Delta^+\sqcup\Delta^-_{\pi'}=\bigsqcup_{\gamma\in S}\Gamma_{\gamma}\sqcup T.$$

\subsection{A basis.}

Recall the set $S\subset\Delta^+\sqcup\Delta^-_{\pi'}$ given in subsection \ref{S}.
\begin{lm}\label{basis}

The set $S$ is such that its restriction to $\h_{\pi'}\oplus\h_{\Lambda}$ is a basis for its dual vector space.

\end{lm}

\begin{proof}

By subsections \ref{rootsp} and \ref{truncCartan} we have that $\dim(\h_{\pi'}\oplus\h_{\Lambda})=n-3+1=n-2$ and one checks that $\lvert S\rvert=n-2$ too. 
Moreover one has that $\h\cap\p_{\Lambda}=\{h\in\h\mid\;h(\Lambda(\p))=\{0\}\}=\h_{\pi'}\oplus\h_{\Lambda}$, namely $\h_{\pi'}\oplus\h_{\Lambda}$ is the orthogonal in $\h$, with respect to the duality, of $\Lambda(\p)\subset\h^*$ and recall Eq. \ref{weightSy(p)}.
Inspired by \cite[Proof of Prop. 5.7]{J5}, it is easily seen that the proof of the Lemma will follow if we can show that
$$(\h_{\pi'}\oplus\h_{\Lambda})^{\perp}+\sum_{\alpha\in S}\mathbb K\alpha=\h^*$$
namely that
\begin{equation}\mathbb K\Lambda(\p)+\sum_{\alpha\in S}\mathbb K\alpha=\h^*.\label{Slibre}\end{equation}

Eq. \ref{Slibre} will follow from the following equality :

\begin{equation} \mathbb Z\pi=\mathbb Z S+\mathbb Z \Lambda(\p).\label{Slibrebis}\end{equation}

Since $\lvert S\rvert=n-2$, if Eq. \ref{Slibre} is true, then it is a direct sum and $$\sum_{\alpha\in S}\mathbb K\alpha=\bigoplus_{\alpha\in S}\mathbb K\alpha.$$

By Eq. \ref{weightclassique} and \ref{weightSy(p)} and the definition of the elements $\beta_i$, we have that
  \begin{equation}\varpi_s+\varpi_{n-s}=\sum_{i=1}^s\beta_i\in\Lambda(\p)\label{sumortho}\end{equation} (see also \cite[Table I]{J1}).

It follows that $\beta_s\in\mathbb ZS+\mathbb Z\Lambda(\p)$.

Now $$\beta_s-\gamma_s=\alpha_{n-s}\in\mathbb ZS+\mathbb Z\Lambda(\p).$$
 If $s=1$ and $a\ge b$ that is, if $n=3$, Eq. \ref{Slibrebis} follows, since in this case $S=\{\gamma_1=\alpha_1\}$.

Assume from now on that $s>1$ and $a\ge b$ or that $a<b$.

We will show by an increasing induction on $k$ that, for any $0\le k\le c-1$, we have 
\begin{equation}\alpha_{n-s+k}\in\mathbb Z S+\mathbb Z\Lambda(\p)\;\hbox{\rm and}\;\;\alpha_{n-s-k}\in\mathbb Z S+\mathbb Z\Lambda(\p).\label{eq}\end{equation}

Assume that $c\ge 2$ and
that for some nonnegative integer $k<c-1$, $\alpha_{n-s-j}\in\mathbb ZS+\mathbb Z\Lambda(\p)$ and $\alpha_{n-s+j}\in\mathbb ZS+\mathbb Z\Lambda(\p)$ for all $0\le j\le k$.

Since $\delta_{1}=\alpha_{n-s-1}+\alpha_{n-s}\in S$, we have $$\alpha_{n-s-1}\in\mathbb Z S+\mathbb Z\Lambda(\p).$$

Since $\beta_{s-1}-\gamma_{s-1}=\alpha_{n-s-1}+\alpha_{n-s}+\alpha_{n-s+1}\in\mathbb ZS$ we obtain that $$\alpha_{n-s+1}\in\mathbb ZS+\mathbb Z\Lambda(\p).$$
Suppose that $c\ge 3$. Since, for $1\le j\le k$, $$\delta_{j+1}-\delta_j=\alpha_{n-s-(j+1)}+\alpha_{n-s+j}\in\mathbb Z S$$ we obtain by the induction hypothesis and the above that
$$\alpha_{n-s-(k+1)}\in\mathbb Z S+\mathbb Z\Lambda(\p).$$ 

Since
\begin{align*}\beta_{s-k-1}-\gamma_{s-k-1}=\ep_{n-s-k-1}-\ep_{n-s+k+2}\\
=\underbrace{\alpha_{n-s-k-1}+\alpha_{n-s-k}+\ldots+\alpha_{n-s+k}}_{\in\mathbb ZS+\mathbb Z\Lambda(\p)\,\hbox{\rm\scriptsize by induction hypothesis and the above}}\!\!+\,\alpha_{n-s+k+1}\in\mathbb ZS\end{align*}
we deduce that $$\alpha_{n-s+k+1}\in\mathbb ZS+\mathbb Z\Lambda(\p).$$

Eq. \ref{eq} follows. In other words one has, if $b\le a$

\begin{equation}\forall s+1\le j\le 2n-3s-1,\;\;\alpha_j\in\mathbb ZS+\mathbb Z\Lambda(\p)\label{eq1}\end{equation}

and if $a<b$,

\begin{equation}\forall n-2s+1\le j\le n-1,\;\;\alpha_j\in\mathbb ZS+\mathbb Z\Lambda(\p).\label{eq2}\end{equation}

Assume now that $2s<n\le 3s\iff a\ge b$. 

We will show that
\begin{equation}\forall\;0\le k\le n-2s-1,\;\;\alpha_{s-k}\in\mathbb ZS+\mathbb Z\Lambda(\p).\label{eq3}\end{equation}

Since
$$\gamma_s=\alpha_s+\underbrace{\sum_{j=s+1}^{n-s-1}\alpha_j}_{\in\mathbb ZS+\mathbb Z\Lambda(\p)\;\hbox{\rm\scriptsize by}\; Eq. \ref{eq1} }\in S$$
we obtain that $\alpha_s\in\mathbb ZS+\mathbb Z\Lambda(\p)$.

Let $1\le k\le n-2s-1$.
Since \begin{equation*} \gamma_{s-k+1}-\gamma_{s-k}=-\alpha_{s-k}+\alpha_{n-s-k}\in\mathbb ZS\end{equation*} and by Eq. \ref{eq1}, we obtain that
$$\alpha_{s-k}\in\mathbb ZS+\mathbb Z\Lambda(\p)$$ which gives Eq. \ref{eq3}.

Combining with Eq. \ref{eq1}, if $a\ge b$, we have that
\begin{equation}\forall \;3s-n+1\le j\le 2n-3s-1,\;\;\alpha_j\in\mathbb ZS+\mathbb Z\Lambda(\p).\label{eq4}\end{equation}

We claim that 
\begin{align} \forall\,0\le k\le [(3s-n-2)/2], \nonumber\\
\alpha_{2n-3s+k}\in\mathbb ZS+\mathbb Z\Lambda(\p)\;\hbox{\rm and}\;\alpha_{3s-n-k}\in\mathbb ZS+\mathbb Z\Lambda(\p).\label{eq5}\end{align}

Our claim implies that $3s-n\ge 2$. Let us assume this. We will show the claim by an increasing induction on $k$.

First step, since
$$\beta_1-\beta'_1-\beta''_1=\alpha_{3s-n}+\underbrace{\sum_{j=3s-n+1}^{2n-3s-1}\alpha_j}_{\in\mathbb ZS+\mathbb Z\Lambda(\p)\;\hbox{\rm\scriptsize by Eq.\ref{eq4}}}\in\mathbb ZS$$

we obtain that $$\alpha_{3s-n}\in\mathbb ZS+\mathbb Z\Lambda(\p).$$

Moreover since 
$$\beta_{3s-n}=\underbrace{\sum_{j=3s-n}^{2n-3s-1}\alpha_j}_{\in\mathbb ZS+\mathbb Z\Lambda(\p)\;\hbox{\rm\scriptsize by the above}}+\alpha_{2n-3s}\in S$$
because $3s-n\le s-1$,
we obtain that $$\alpha_{2n-3s}\in\mathbb ZS+\mathbb Z\Lambda(\p).$$

Assume that, for some nonnegative integer $k$ such that $k<[(3s-n-2)/2]$ we have,
for all nonnegative integer $j$ such that $j\le k$, $$\alpha_{2n-3s+j}\in\mathbb ZS+\mathbb Z\Lambda(\p)\;\hbox{\rm and}\;\alpha_{3s-n-j}\in\mathbb ZS+\mathbb Z\Lambda(\p).$$ Since

\begin{align*}\beta_{k+2}-\beta'_{k+2}-\beta''_{k+2}=&\alpha_{3s-n-k-1}\\
&+\underbrace{\sum_{j=3s-n-k}^{2n-3s+k}\alpha_j}_{\in\mathbb ZS+\mathbb Z\Lambda(\p)\,\hbox{\rm\scriptsize by induction hypothesis and Eq.\ref{eq4} }}\in\mathbb ZS\end{align*}
we obtain that $$\alpha_{3s-n-k-1}\in\mathbb ZS+\mathbb Z\Lambda(\p).$$

Moreover since $$\beta_{3s-n-k-1}=\underbrace{\sum_{j=3s-n-k-1}^{2n-3s+k}\alpha_j}_{\in\mathbb ZS+\mathbb Z\Lambda(\p)\,\hbox{\rm\scriptsize by the above}}+\alpha_{2n-3s+k+1}\in S$$ 
because $1\le 3s-n-k-1\le s-1$, we obtain that
$$\alpha_{2n-3s+k+1}\in\mathbb ZS+\mathbb Z\Lambda(\p)$$ and the claim (Eq. \ref{eq5}) follows.

Combining with Eq. \ref{eq4} and using Eq. \ref{partentier} we obtain, for $a\ge b$
\begin{equation}\forall [(3s-n+3)/2]\le j\le[(3n-3s-2)/2],\,\alpha_j\in\mathbb ZS+\mathbb Z\Lambda(\p).\label{eq6}\end{equation}

We also show by a decreasing induction on $k$, using $\beta_{k-1}-\beta'_{k-1}$ for $k\ge 2$, that
\begin{equation*}\forall\,1\le k\le[(3s-n+3)/2],\,\alpha_k\in\mathbb ZS+\mathbb Z\Lambda(\p).\end{equation*}

Combining with Eq. \ref{eq6}  we obtain, for $a\ge b$, that
\begin{equation*}\forall\,1\le k\le [(3n-3s-2)/2],\,\alpha_k\in\mathbb ZS+\mathbb Z\Lambda(\p).\end{equation*}

Finally  one shows that 
\begin{equation*} \forall [(3n-3s-2)/2]\le k\le n-1, \alpha_k\in\mathbb ZS+\mathbb Z\Lambda(\p)\end{equation*}
by using an increasing induction on $k$ and $\beta_{n-k-1}$ since $1\le n-k-1\le s-1$ if $[(3n-3s-2)/2]\le k< n-1$. Eq. \ref{Slibrebis} follows for $a\ge b$ (equivalently $2s<n\le 3s$).

To complete the proof, consider now the case $n>3s$ (equivalently $a<b$) and recall Eq. \ref{eq2}.

Since
$$\delta_s=\ep_{n-2s}-\ep_n=\alpha_{n-2s}+\underbrace{\sum_{j=n-2s+1}^{n-1}\alpha_j}_{\in\mathbb ZS+\mathbb Z\Lambda(\p)\,\hbox{\rm\scriptsize by Eq. \ref{eq2}}}\in S$$
we obtain $$\alpha_{n-2s}\in\mathbb ZS+\mathbb Z\Lambda(\p).$$
Let $1\le k\le s-1$. Since $$\gamma_{s-k+1}-\gamma_{s-k}=\ep_{s-k+1}-\ep_{n-s-k+1}-\ep_{s-k}+\ep_{n-s-k}=-\alpha_{s-k}+\alpha_{n-s-k}\in\mathbb Z S$$

and since $\alpha_{n-s-k}\in\mathbb ZS+\mathbb Z\Lambda(\p)$ by Eq. \ref{eq2}, we deduce that $$\alpha_{s-k}\in\mathbb ZS+\mathbb Z\Lambda(\p).$$ In other words one has, for all $1\le k\le s-1$, that $$\alpha_k\in\mathbb ZS+\mathbb Z\Lambda(\p).$$ 

Then using $\beta'_{k+1}-\beta''_{k+1}=\alpha_{n-2s-(k+1)}$ and also $\beta'_{k+1}-\beta'_{k+2}=\alpha_{s+k+1}+\alpha_{n-2s-k-1}$ when this is well defined (or simply $\beta'_1=\alpha_{s+1}$ in case when $n=3s+2$), one can show inductively that
$$\forall \;s+1\le k\le n-2s-1,\,\alpha_k\in\mathbb ZS+\mathbb Z\Lambda(\p).$$

Since $\gamma_s=\alpha_s+\ldots+\alpha_{n-s-1}$ we obtain finally that $\alpha_s\in\mathbb ZS+\mathbb Z\Lambda(\p)$.

Eq. \ref{Slibrebis} follows in case $a<b$ too. The proof of the Lemma is complete.
\end{proof}

 \subsection{The nondegeneracy of the restriction of the skew-symmetric bilinear form.}\label{nondeg}
 
Recall the notation in Lemma \ref{lmAP}. To prove that the restriction to $\g_O\times\g_O$ of the bilinear form $\widetilde\Phi_y$ is nondegenerate, we will decompose the set $S$ into $S=S^m\sqcup S^+\sqcup S^-$
 where $S^+$, resp. $S^-$, is formed by the roots $\gamma\in S$ such that the Heisenberg set $\Gamma_{\gamma}\subset\Delta^+$, resp. $\Gamma_{\gamma}\subset\Delta^-_{\pi'}$ and $S^m$ is formed by the roots $\gamma\in S$ such that the Heisenberg set $\Gamma_{\gamma}$ contains both positive and negative roots.
 We will set $$O^m=\bigsqcup_{\gamma\in S^m}\Gamma^0_{\gamma},\;\;O^+=\bigsqcup_{\gamma\in S^+}\Gamma^0_{\gamma},\;\;O^-=\bigsqcup_{\gamma\in S^-}\Gamma^0_{\gamma}.$$
 
 Our choice of set $S$ (see subsection \ref{S}) and of disjoint Heisenberg sets with centre belonging to $S$ (see subsection \ref{ST}) implies the following.
 
Assume first that $a\ge b$. 
We have 
 $$S^+=\{\beta_i;\;1\le i\le s-1\}$$ and by Eq. \ref{decompSQ}
 \begin{align}O^+=\bigsqcup_{1\le i\le s-1}\Gamma_{\beta_i}^0=\big(SQ\setminus\{\beta_i\}_{1\le i\le s}\big)\sqcup B_1^+\sqcup B_3^+.\label{decomp1Oplus}\end{align}
 We have $$S^-=\{-\beta'_k;\;1\le k\le[(3s-n+1)/2],\;-\beta''_{\ell};\;1\le\ell\le[(3s-n)/2]\}$$
 and
 \begin{align}O^-=&\bigsqcup_{1\le k\le[(3s-n+1)/2]}\Gamma_{-\beta'_k}^0\,\sqcup\bigsqcup_{1\le k\le[(3s-n)/2]}\Gamma_{-\beta''_k}^0\nonumber\\
 =&\,\big(B_{3,\,r}^-\setminus\{-\beta'_k\}\big)\sqcup\big(B_{1,\,r}^-\setminus\{-\beta''_k\}\big)\label{decomp1Omoins}\end{align}
 by Eq. \ref{remfirst} and \ref{remsecond}.
 We have
$$ S^m=\{\gamma_{s-j};\;0\le j\le n-2s-1,\;\delta_k;\;1\le k\le n-2s-1\}$$ and
\begin{align} O^m=&\bigsqcup_{0\le j\le n-2s-1}\Gamma_{\gamma_{s-j}}^0\,\sqcup\bigsqcup_{1\le k\le n-2s-1}\Gamma_{\delta_k}^0\nonumber\\
=&\Big(R_1\setminus\{\gamma_{s-j}\}\Big)\sqcup\Big(B_1^-\setminus B_{1,\,r}^-\Big)\sqcup\Big(R_2\setminus\big(\{\delta_k\}\sqcup (T\cap R_2)\big)\Big)\nonumber\\
&\sqcup \Big(B_3^-\setminus B_{3,\,r}^-\Big)\sqcup \Big(B_2\setminus (T\cap B_2)\Big).\label{decomp1Om}\end{align}

 Assume $a<b$. We have
  $$S^+=\{\beta_i;\;1\le i\le s-1,\,\beta'_k;\;1\le k\le[(n-3s)/2]\}$$ and
  \begin{align}O^+=&\bigsqcup_{1\le i\le s-1}\Gamma_{\beta_i}^0\,\,\sqcup\bigsqcup_{1\le k\le[(n-3s)/2]}\Gamma_{\beta'_k}^0\nonumber\\
  =&\big(SQ\setminus\{\beta_i\}_{1\le i\le s}\big)\sqcup B_1^+\sqcup B_3^+\sqcup \big(B_{2,\,r}^+\setminus\{\beta'_k\}\big)\label{decomp2Oplus}\end{align}
  by Eq. \ref{decompSQ} and Eq. \ref{remfourth}.
We have $$S^-=\{-\beta''_{\ell};\;1\le\ell\le[(n-3s-1)/2]\}$$ and
 \begin{align}O^-=&\bigsqcup_{1\le k\le[(n-3s-1)/2]}\Gamma_{-\beta''_k}^0\nonumber\\
 =&\,B_{2,\,r}^-\setminus\{-\beta''_k\}\label{decomp2Omoins}\end{align}
 by Eq. \ref{remthird}.
 We have 
 $$S^m=\{\gamma_{s-j};\;0\le j\le s-1,\;\delta_k;\;1\le k\le s\}$$ and
\begin{align} O^m=&\bigsqcup_{0\le j\le s-1}\Gamma_{\gamma_{s-j}}^0\sqcup\bigsqcup_{1\le k\le s}\Gamma_{\delta_k}^0\nonumber\\
=&\Big(R_1\setminus\{\gamma_{s-j}\}\Big)\sqcup \Big(B_2^+\setminus B_{2,\,r}^+\Big)\sqcup B_1^-\sqcup\Big(R_2\setminus\big(\{\delta_k\}\sqcup (T\cap R_2)\big)\Big)\nonumber\\
&\sqcup B_3^-\sqcup\Big(B_2^-\setminus \big(B_{2,\,r}^-\sqcup (T\cap B_2^-)\big)\Big).\label{decomp2Om}\end{align}

 Recall the set $S_{\alpha}$ (Eq. \ref{Salph} ) for every root $\alpha\in O$. Since condition $({\bf C})$ is satisfied, we know that for every $\alpha\in O$, $\theta(\alpha)\in S_{\alpha}$. For any positive integer $k$
 we set, as in \cite[Sec. 4]{F01}, 
$$ O_k=\{\alpha\in O;\;\lvert S_{\alpha}\rvert=k\}.$$

As we already said, if $O=O_1$ then the restriction to $\g_O\times\g_O$ of the bilinear form $\widetilde\Phi_y$ is nondegenerate. Unfortunately it is not the case in general for our choice of Heisenberg sets.

 As in \cite[Sec. 4]{F01}, together with condition $({\bf C})$, we will require  the following condition $({\bf C'})$.
  
 {\bf Condition $\rm\bf (C')$}:
 \begin{align*}\begin{cases} O=O_1\sqcup O_2\sqcup O_3\\
 \alpha\in O_3\Longrightarrow \exists\,\widetilde\alpha\in S_{\alpha}\cap O_2\;;\;\theta\big(\widetilde\alpha\big)\in O_1\end{cases}\end{align*}

 In Lemma \ref{O} we claim that, in case $a<b$, we have $O=O_1\sqcup O_2$ and that, in case $a\ge b$, we have $O=O_1\sqcup O_2\sqcup O_3$. Moreover  in Lemma \ref{ConditionC'} we claim that, in case $a\ge b$, the last condition of $({\bf C'})$ is satisfied for any root $\alpha\in O_3$. 

 With conditions $({\bf C})$ and $({\bf C'})$, for any root $\alpha\in O$ satisfying some further condition (see below) we can define as in \cite[Sec.4]{F01} sequences in $O$ constructed from $\alpha$ or from $\theta(\alpha)$ and then we can define a socalled {\it stationary root}. Here are the following definitions.
 
 \begin{defi}(\cite[Def. 4.3.2, 4.3.3, 4.3.5]{F01})\label{defirootstat}
 Let $\alpha\in O$ for which there exists no $\beta\in O_3$ such that $\alpha=\widetilde\beta$ or $\alpha=\theta\big(\widetilde\beta\big)$. 
 \begin{enumerate}
\item  We define the sequence $(\alpha^i)_{i\in\mathbb N}$ in $O$ constructed from the root $\alpha$  inductively as follows. We set $\alpha^0=\alpha$ and for any $i\in\mathbb N$, 
 \begin{enumerate}
\item if $\theta(\alpha^i)\in O_1$, then we set $\alpha^{i+1}=\alpha^i$. 
\item if $\theta(\alpha^i)\in O_2$, then $\alpha^{i+1}$ is the only root in $S_{\theta(\alpha^i)}\setminus\{\alpha^i\}$. 
\item if $\theta(\alpha^i)\in O_3$ then $\alpha^{i+1}$ is the only root in $S_{\theta(\alpha^i)}\setminus\big\{\alpha^i,\,\widetilde{\theta(\alpha^i)}\big\}$.
 \end{enumerate}
\item  Similarly we define the sequence $(\alpha^{(i)})_{i\in\mathbb N}$  in $O$ constructed from the root $\theta(\alpha)\in O$. In other words, for any nonnegative integer $i$, we set $\alpha^{(i)}=\theta(\alpha)^i$.

\item We say that the sequence $(\alpha^i)_{i\in\mathbb N}$ is stationary if there exists a nonnegative integer $i_0$ such that $\alpha^{i_0+1}=\alpha^{i_0}$.

\item We say that the root $\alpha$ is a stationary root if the sequences $(\alpha^i)_{i\in\mathbb N}$ and $(\alpha^{(i)})_{i\in\mathbb N}$ are both stationary.
 \end{enumerate}
 \end{defi}
 
 We refer to \cite[Fig. 1, 2]{F01} for an illustration of such sequences. Recall the following propositions.
 
 \begin{prop}(\cite[Lem. 4.3.8]{F01})\label{proprootstat}
Let $\alpha\in O$ verifying the  hypotheses of Definition \ref{defirootstat} and set $C_{\alpha}=\{\alpha^i,\,\theta(\alpha^i),\,\alpha^{(i)},\,\theta(\alpha^{(i)});\;i\in\mathbb N\}$ which is called the chain passing through $\alpha$. Then any root in the chain $C_{\alpha}$ also verifies the  hypotheses of Definition \ref{defirootstat}. Moreover if one of the two sequences $(\alpha^i)_{i\in\mathbb N}$ or $(\alpha^{(i)})_{i\in\mathbb N}$ is stationary, then  the other one is also stationary, which implies that $\alpha$ is a stationary root. Moreover if $\alpha$ is a stationary root then any root in the chain $C_{\alpha}$ is also a stationary root.
 \end{prop}

 \begin{prop}{(\cite[Prop. 4.4.1]{F01})}\label{propnondeg}
 We assume that:
 
 \begin{enumerate}
 \item[(1)] $S_{\h_{\pi'}\oplus\h_{\Lambda}}$ is a basis for $(\h_{\pi'}\oplus\h_{\Lambda})^*$.\smallskip
 
 \item[(2)] If $\alpha\in O^+$ then $S_{\alpha}\cap O^+=\{\theta(\alpha)\}$.\smallskip
 
 \item[(3)] If $\alpha\in O^-$ then $S_{\alpha}\cap O^-=\{\theta(\alpha)\}$.\smallskip
 
 \item[(4)] If $\alpha\in O^m$ then $\alpha$ is a stationary root.
  \end{enumerate}
 
 Let $y=\sum_{\gamma\in S}x_{-\gamma}$ and $\widetilde\Phi_y$ be the skew-symmetric bilinear form defined by $\widetilde\Phi_y(x,\,x')=K(y,\,[x,\,x']_{\widetilde\p})$ for all $x,\,x'\in\widetilde\p$.
 
 Then the restriction to $\g_O\times\g_O$ of $\widetilde\Phi_y$ is nondegenerate.

 \end{prop}
 
 \subsection{Conditions of the Proposition of nondegeneracy.}
 Recall the notation of the previous subsections.
 By Lemma \ref{basis} we already know that condition (1) of Proposition \ref{propnondeg} is satisfied. Moreover one can remark that, for a root $\alpha\in O$, if there exists only one root in $S$ on the same row and no root in $S$ on the same column, or vice versa, then $\alpha\in O_1$. If there exist two roots in $S$ on the same row or column then $\alpha\in O_1\sqcup O_2$. If there exist  three roots in $S$ on the same row or column, then $\alpha\in O_1\sqcup O_2\sqcup O_3$. We check below  that $O= O_1\sqcup O_2\sqcup O_3$ in case $a\ge b$, and  that $O= O_1\sqcup O_2$ in case $a<b$. 
 
 Let $\alpha\in(\Delta^+\setminus\Delta^+_{\pi'})\cap O$. Then if $\beta\in S_{\alpha}$ we have necessarily that $\beta\in\Delta_{\pi'}$ in view of subsection \ref{defcontp} and the definition of $S_{\alpha}$ (see Eq. \ref{Salph}) and then $\alpha+\beta\in S\cap(\Delta^+\setminus\Delta^+_{\pi'})$. Moreover by Eq. \ref{expSQ}, \ref{expR1} and  \ref{expR2},  if $\alpha\in R_i$ ($i=1,\,2$) then $\alpha+\beta\in R_i$ and if $\alpha\in SQ$ then $\alpha+\beta\in SQ$. Then for a root $\alpha\in O\cap R_i$, resp. $\alpha\in O\cap SQ$, we only consider the roots in $S\cap R_i$, resp. in $S\cap SQ$,  on the same row or column as $\alpha$ for determining the cardinality of $S_{\alpha}$. 
 
 Observing Figures \ref{Fig1} and \ref{Fig2} of subsection \ref{ST}, we then can claim (in both cases $a\ge b$ or $a<b$) that $O\cap(\Delta^+\setminus\Delta^+_{\pi'})\subset O_1\sqcup O_2$. Moreover, in case $a<b$, we claim that $O\cap\Delta_{\pi'}\subset O_1\sqcup O_2$ and in case $a\ge b$ we claim that $O\cap(\Delta_{\pi'_2}\sqcup\Delta^+_{\pi'_1}\sqcup\Delta^+_{\pi'_3})\subset O_1\sqcup O_2$. Let us give some explanations. Consider firstly case $a\ge b$. A root $\alpha\in O\cap\Delta^+_{\pi'_1}$ cannot lie at the same time on the same row and on the same column as a root  in $S\cap B_1^-$. Moreover no remaining root in $B_{1,\,r}^-$  meets a root $\gamma_{s-j}$ (for $0\le j\le b-1$) on its row. Then $\alpha\in O\cap\Delta^+_{\pi'_1}$ can meet  on its row or column at most two roots in $S$ (one $-\beta''_k$ and one $\beta_{\ell}$, or one $\gamma_{s-j}$ and one $\beta_{\ell}$).  Hence such a root $\alpha\in O_1\sqcup O_2$.  Similarly a root $\alpha\in O\cap\Delta^+_{\pi'_3}$ cannot lie both on the same column and on the same row as a root in $S\cap B_3^-$. Moreover no remaining root in $B_{3,\,r}^-$  meets on its column a root $\delta_k$, for $1\le k\le b-1$. It follows that $\alpha\in O\cap\Delta^+_{\pi'_3}$ can meet on its row or column at most two roots in $S$ (one $-\beta'_k$ and one $\beta_{\ell}$ or one $\delta_k$ and one $\beta_{\ell}$). Finally a root $\alpha\in O\cap\Delta_{\pi'_2}$ meets on its row or column at most two roots in $S$ (one $\gamma_{s-j}$ and one $\delta_k$).  
 
 Consider now case $a<b$. Let $\alpha\in O\cap\Delta_{\pi'}$ and observe Fig. \ref{Fig2} of subsection \ref{ST}. If $\alpha\in\Delta_{\pi'_1}\sqcup\Delta_{\pi'_3}$, then the row or the column of $\alpha$  meets at most two roots in $S$ (actually in $S\cap(\Delta^+\setminus\Delta^+_{\pi'})$). Hence $\alpha\in O_1\sqcup O_2$ in this case. If $\alpha\in\Delta_{\pi'_2}$, then again $\alpha\in O_1\sqcup O_2$. It comes from the fact that the remaining roots in $B_{2,\,r}^-$ or in $B_{2,\,r}^+$ have their column lying on the left of the $\gamma_{s-j}$ (for $0\le j\le s-1$) and their row lying above the $\delta_k$ (for $1\le k\le s$).  Moreover there is no root in $S\cap B_{2,\,r}^-$  having the same row, resp. the same column, as a root in $S\cap B_{2,\,r}^+$. Hence if $\alpha\in O\cap(B_2^+\setminus B_{2,\,r}^+)$ then the row and the column of $\alpha$ meets in $S$ at most two roots : either one root in $S\cap B_{2,\,r}^+$ and one $\gamma_{s-j}$, $0\le j\le s-1$, or one root in $S\cap B_{2,\,r}^-$ and one $\gamma_{s-j}$, $0\le j\le s-1$, or one $\gamma_{s-j}$, $0\le j\le s-1$, and one $\delta_k$, $1\le k\le s$. If $\alpha\in O\cap B_{2,\,r}^+$, then again the row and the column of $\alpha$ meets at most two roots in $S$ : either two roots in $S\cap B_{2,\,r}^+$ or one root in $S\cap B_{2,\,r}^+$ and one root in $S\cap B_{2,\,r}^-$. Similar considerations hold for a root $\alpha\in O\cap B_2^-$.
 
 Finally for case $a\ge b$, we can observe on Figure \ref{Fig1} of subsection \ref{ST} that $O\cap(\Delta^-_{\pi'_1}\cup\Delta^-_{\pi'_3})\subset O_1\sqcup O_2\sqcup O_3$. 
 Summarizing the above observations, we obtain the following Lemma.
 
 \begin{lm}\label{O}
 In both cases ($a\ge b$ or $a<b$) we have that $O\cap(\Delta^+\setminus\Delta^+_{\pi'})\subset O_1\sqcup O_2$ and in case $a<b$ we even have that $O=O_1\sqcup O_2$.
 Moreover in case $a\ge b$ we have that $O\cap(\Delta_{\pi'_2}\sqcup\Delta^+_{\pi'_1}\sqcup\Delta^+_{\pi'_3})\subset O_1\sqcup O_2$ and $O=O_1\sqcup O_2\sqcup O_3$.
 
 \end{lm}
 
 It remains to verify that the last condition of $({\rm \bf C'})$ is satisfied for case $a\ge b$. It is the following lemma.
 
 \begin{lm}\label{ConditionC'}
 
 In case $a\ge b$, the last condition of $({\rm \bf C'})$ is satisfied for any root $\alpha\in O_3$.
 
 \end{lm}
 
 \begin{proof}
 
 Consider case $a\ge b$ and
 let us give firstly the expansion of the roots in $\Delta^-_{\pi'_1}\sqcup\Delta^-_{\pi'_3}$. Roots in $\Delta^-_{\pi'_1}$ are of the form :
 $\ep_j-\ep_i$ with $1\le i<j\le s$ and roots in $\Delta^-_{\pi'_3}$ are of the form :
 $\ep_j-\ep_i$ with $n-s+1\le i<j\le n$. We will show the statement for a root $\gamma\in O_3\cap\Delta^-_{\pi'_1}$, the proof for roots in $O_3\cap\Delta^-_{\pi'_3}$ being very similar.
 
 Consider $\gamma=\ep_j-\ep_i$ with $1\le i<j\le 3s-n$, $\gamma\in O_3$, that is, we firstly consider a root which lies in the set $B_{1,\,r}^-$ of remaining roots in the upper left corner of $B_1^-$. Recall that the elements in $S\cap\Delta^-_{\pi'_1}$ are the $-\beta''_k$, $1\le k\le[(3s-n)/2]$, with $\beta''_k=\ep_k-\ep_{3s-n+1-k}$. Suppose that there exist $k_1,\,k_2$, $1\le k_1,\,k_2\le[(3s-n)/2]$, such that $i=k_1$ and $j=3s-n+1-k_2$, namely that $\gamma$ lies on the same row as $-\beta''_{k_2}$ and on the same column as $-\beta''_{k_1}$. Moreover $\gamma$ also lies on the same row as $\beta_j$ and  there exists no $0\le\ell\le n-2s-1$ such that $\gamma$  lies on the same row as $\gamma_{s-\ell}$ since $\gamma$ lies in $B_{1,\,r}^-$ and then does not meet on its row any $\gamma_{s-\ell}$. Hence we have indeed that $\gamma\in O_3$ and $S_{\gamma}=\{\ep_i-\ep_{3s-n+1-j},\,\ep_{3s-n+1-i}-\ep_j,\,\ep_i-\ep_{n+1-j}\}$. Actually one can set $\widetilde\gamma=\ep_i-\ep_{n+1-j}\in\Gamma_{\beta_i}^0$ which lies in $O_2$ since $\widetilde\gamma\in SQ$.  Moreover $\theta(\widetilde\gamma)=\ep_{n+1-j}-\ep_{n+1-i}\in\Delta^+_{\pi'_3}$. One can then check that $\theta(\widetilde\gamma)$ cannot lie on the same row or column as a $-\beta'_k$, $1\le k\le [(3s-n+1)/2]$, nor on the same column as a $\delta_k$, $1\le k\le n-2s-1$. Hence $\theta(\widetilde\gamma)$ only lies on the same column as $\beta_i$ and then $\theta(\widetilde\gamma)\in O_1$.
 
 Similarly suppose now that $\gamma=\ep_j-\ep_i$ with $1\le i\le s-1$, $3s-n+1\le j\le s$ and $i<j$ and that $\gamma\in O_3$. Then $\gamma$ lies on the same column as $-\beta''_i$ (and then $1\le i\le[(3s-n)/2]$) and on the same row as $\gamma_j$ and $\beta_j$. As above, we have that $\widetilde\gamma=\ep_i-\ep_{n+1-j}\in(\Delta^+\setminus\Delta^+_{\pi'})\cap SQ$ and hence $\widetilde\gamma\in O_2$. Moreover $\theta(\widetilde\gamma)=\ep_{n+1-j}-\ep_{n+1-i}\in\Delta^+_{\pi'_3}$ and one can check that $\theta(\widetilde\gamma)$ does not lie on the same row or column as a $-\beta'_k$ nor on the same column as a $\delta_k$. Hence $\theta(\widetilde\gamma)\in O_1$.
 \end{proof}
 
 \begin{Rq}\rm If a root $\gamma\in O_3\cap \Delta^-_{\pi'_3}$ then $\gamma=\ep_j-\ep_i$ with $n+1-s\le i<j\le n$, $\widetilde\gamma=\ep_{n+1-i}-\ep_j\in SQ\cap O_2$ and $\theta(\widetilde\gamma)=\ep_{n+1-j}-\ep_{n+1-i}\in\Delta^+_{\pi'_1}\cap O_1$.
 \end{Rq}
 
 Now we can check condition (4) of Proposition \ref{propnondeg}.
 \begin{lm}\label{lmstatroot}
 Every root $\alpha\in O^m$ is a stationary root.
 \end{lm}
 
 \begin{proof}

By the description of the Heisenberg sets in subsection \ref{ST}, it suffices to prove the Lemma for a root $\alpha\in O^m\cap(\Delta^+\setminus\Delta^+_{\pi'})$, since it is then also true for $\theta(\alpha)\in O^m\cap\Delta_{\pi'}$ by Prop. \ref{proprootstat}. Moreover by Eq. \ref{decomp1Om} and Eq. \ref{decomp2Om}, it suffices to prove the lemma for a root $\alpha\in O^m\cap R_i$ for $i=1,\,2$. If $\alpha\in O_1\cap R_i$ then $\alpha$ is stationary, since in this case $\alpha^{(1)}=\theta(\alpha)=\alpha^{(0)}$ by Definition \ref{defirootstat} and Proposition \ref{proprootstat}.
 
 Assume that $a\ge b$. 
Take a root $\alpha\in O_2\cap R_1$ that is, $\alpha$ lies on the same row  as a $\gamma_{s-j}$ and on the same column as a $\gamma_{s-\ell}$ for $0\le j\neq \ell\le n-2s-1$.
  If $j>\ell$ then $\alpha\in\Gamma^0_{\gamma_{s-\ell}}$ and $\alpha$ lies on the same row and on the right of $\gamma_{s-j}$ and on the same column and above $\gamma_{s-\ell}$. If $j<\ell$ then $\alpha\in\Gamma^0_{\gamma_{s-j}}$ and $\alpha$ lies on the same row and on the left of $\gamma_{s-j}$ and on the same column and below $\gamma_{s-\ell}$.
 
 We will suppose $j<\ell$, the case $j>\ell$ being very similar. In other words $$\alpha=\ep_{s-j}-\ep_{n-s-\ell}\in\Gamma^0_{\gamma_{s-j}}\cap R_1.$$
 
Since $\alpha\in\Gamma^0_{\gamma_{s-j}}$, we have  $\theta(\alpha)=\ep_{n-s-\ell}-\ep_{n-s-j}\in \Delta^+_{\pi'_2}$ and $\theta(\alpha)$ lies on the same row (and on the left) of $\delta_{\ell}=\ep_{n-s-\ell}-\ep_{n-s+\ell}$.
 
 The only root $\beta\neq\alpha$ such that $\beta+\theta(\alpha)\in S$ is such that $\beta+\theta(\alpha)=\delta_{\ell}$, and then $\beta=\ep_{n-s-j}-\ep_{n-s+\ell}$. 
 If $j=0$ then $\beta\in T$. It follows that $\theta(\alpha)\in O_1$ and $\alpha^1=\alpha^0$ that is, $\alpha$ is a stationary root by Definition \ref{defirootstat} and Proposition \ref{proprootstat}.
 
 If $j\ge 1$, then $\alpha^1=\beta\in R_2\cap \Gamma^0_{\delta_j}$ since $\alpha^1$ lies on the same row and on the right of $\delta_j$. Then $\theta(\alpha^1)=\ep_{n-s+\ell}-\ep_{n-s+j}\in\Delta^-_{\pi'_3}$.
 One then can check that $\theta(\alpha^1)$ does not lie on the same row or column as a $-\beta'_{k}$, for any $1\le k\le[(3s-n+1)/2]$, since $\theta(\alpha^1)$ lies in $B_3^-$ on the left and above the bottom right corner $B_{3,\,r}^-$ of remaining roots. If $j=1$ then $\theta(\alpha^1)$ lies on the same column as $\beta_s$ which does not lie in $S$ (and of course $\theta(\alpha^1)$ lies on the same column as $\delta_1$). Hence $\theta(\alpha^1)\in O_1$ in this case, since $S_{\theta(\alpha^1)}=\{\alpha^1\}$. It follows that $\alpha^2=\alpha^1$ and then $\alpha$ is a stationary root by Definition \ref{defirootstat} and Proposition \ref{proprootstat}.
 
 Now assume that $j\ge 2$. Then $\theta(\alpha^1)$ lies on the same column as $\beta_{s+1-j}=\ep_{s+1-j}-\ep_{n-s+j}$ and then $\alpha^2=\ep_{s+1-j}-\ep_{n-s+\ell}\in SQ\cap \Gamma^0_{\beta_{s+1-\ell}}$ since $\alpha^2$ lies on the same column as $\beta_{s+1-\ell}$ and below it. Then $\theta(\alpha^2)=\ep_{s+1-\ell}-\ep_{s+1-j}\in\Delta^+_{\pi'_1}$. But now $\theta(\alpha^2)$ lies on the same row as $\gamma_{s-(\ell-1)}$. It follows that $$\alpha^3=\ep_{s-(j-1)}-\ep_{n-s-(\ell-1)}\in\Gamma^0_{\gamma_{s-(j-1)}}\cap R_1.$$ We have then obtained a new root of the same form as $\alpha$ with $j$ and $\ell$ replaced respectively by $j-1$ and $\ell-1$. Using what we have shown for $j=1$ above, one can conclude by induction that the root $\alpha$ is stationary.
 
 Now assume that $a\ge b$ and take a root $\alpha\in O_2\cap R_2$ that is, $\alpha$ lies on the same row as a $\delta_j$ and on the same column as a $\delta_{\ell}$ for $1\le j\neq\ell\le n-2s-1$.  
 We will suppose that $j<\ell$, the case $j>\ell$ being very similar. In other words, one has
 $$\alpha=\ep_{n-s-j}-\ep_{n-s+\ell}\in\Gamma^0_{\delta_j}\cap R_2.$$
More precisely $\alpha$ lies on the same row and on the right of $\delta_j$ and on the same column and below $\delta_{\ell}$. Then $\theta(\alpha)=\ep_{n-s+\ell}-\ep_{n-s+j}\in\Delta^-_{\pi'_3}$. One can check that there exists no positive integer $k\le[(3s-n+1)/2]$ such that $\theta(\alpha)$ lies on the same row or column as $-\beta'_k$, since $\theta(\alpha)$ lies on the left and above the set $B_{3,\,r}^-$ of remaining roots in $B_3^-$.
 
 If $j=1$, then $\theta(\alpha)$ lies on the first column of $B_3^-$ and then $\theta(\alpha)\in O_1$, hence $\alpha^1=\alpha$ and the root $\alpha$ is stationary. Now suppose that $j\ge 2$. Then $\theta(\alpha)$ lies also on the same column as $\beta_{s+1-j}$ and then $\alpha^1=\ep_{s+1-j}-\ep_{n-s+\ell}\in SQ\cap\Gamma^0_{\beta_{s+1-\ell}}$. Then $\theta(\alpha^1)=\ep_{s+1-\ell}-\ep_{s+1-j}\in\Delta^+_{\pi'_1}$ and one checks again that there exists no positive integer $k\le[(3s-n)/2]$ such that $\theta(\alpha^1)$ lies on the same row or on the same column as any $-\beta''_k$, since $\theta(\alpha^1)$ lies on the right and below the upper left corner $B_{1,\,r}^-$ of remaining roots in $B_1^-$.
 One may observe that $\theta(\alpha^1)$ is also on the same row as $\gamma_{s-(\ell-1)}$ and then $\alpha^2=\ep_{s+1-j}-\ep_{n-s-\ell+1}\in R_1\cap\Gamma^0_{\gamma_{s-(j-1)}}$. Then $\theta(\alpha^2)=\ep_{n-s-\ell+1}-\ep_{n-s-j+1}\in\Delta^+_{\pi'_2}$, which lies on the same row as $\delta_{\ell-1}$. Hence one has
 $$\alpha^3=\ep_{n-s-(j-1)}-\ep_{n-s+(\ell-1)}\in R_2\cap\Gamma^0_{\delta_{j-1}}.$$
 Using the case $j=1$ we deduce by induction that $\alpha$ is a stationary root.
 
 The case $a<b$ can be treated in a same manner, replacing $0\le j\neq\ell\le n-2s-1$ by $0\le j\neq\ell\le s-1$ in the case when $\alpha\in O_2\cap R_1$ and $1\le j\neq\ell\le n-2s-1$ by $1\le j\neq\ell\le s$ in the case when $\alpha\in O_2\cap R_2$. This completes the proof of the Lemma. 
 \end{proof}
 
 Finally we check that conditions (2) and (3) of Prop. \ref{propnondeg} are satisfied.
 \begin{lm}\label{lmtheta}
 
 Let $\alpha\in O^\pm$. Then $S_{\alpha}\cap O^\pm=\{\theta(\alpha)\}$.
 
 \end{lm}
 
 \begin{proof}
 
 Assume firstly that $a\ge b$ and recall Eq. \ref{decomp1Oplus}. Let $\alpha\in O^+\cap SQ$ and $\beta\in S_{\alpha}\cap O^+$. Then $\beta\in B_1^+\sqcup B_3^+$ and $\alpha+\beta\in SQ\cap S=\{\beta_k\}_{1\le k\le s-1}$. By Eq. \ref{decomp1Oplus}, there exist positive integers $1\le i,\,j\le s-1$ such that $\alpha\in\Gamma^0_{\beta_i}\subset{\rm H}^0_{\beta_i}$ and $\beta\in\Gamma^0_{\beta_j}\subset{\rm H}^0_{\beta_j}$. Let $k$ be the positive integer such that $\alpha+\beta=\beta_k$. By Lemma \ref{strong} (iv) if $i\le j$ then $k=i$ and $\beta=\beta_i-\alpha=\theta(\alpha)$ and if $i\ge j$ then $k=j$ and $\alpha=\beta_j-\beta=\theta(\beta)$. In both cases, since $\theta$ is an involution, we obtain that $\beta=\theta(\alpha)$. 
 
 Suppose now that $\alpha\in O^+\cap(B_1^+\sqcup B_3^+)$ and let $\beta\in S_{\alpha}\cap O^+$. Then $\beta$ cannot belong to $B_1^+\sqcup B_3^+$ since there exists no root in $S\cap(B_1^+\sqcup B_3^+)$. Then $\beta\in O^+\cap SQ$ and $\alpha\in S_{\beta}\cap O^+$. We conclude by the above.
 
Suppose that $\alpha\in O^-$ and let $\beta\in S_{\alpha}\cap O^-$. Recall Eq. \ref{decomp1Omoins}. Since $\alpha+\beta\in\Delta$, it is not possible that $\alpha\in B_{3,\,r}^-$ and $\beta\in B_{1,\,r}^-$ or vice-versa. Hence there exist positive integers $i,\,j$, $1\le i,\,j\le [(3s-n+1)/2]$, such that $\alpha\in\Gamma^0_{-\beta'_i}$ and $\beta\in\Gamma^0_{-\beta'_j}$ or there exist positive integers $i,\,j$, $1\le i,\,j\le[(3s-n)/2]$, such that $\alpha\in\Gamma^0_{-\beta''_i}$ and $\beta\in\Gamma^0_{-\beta''_j}$. In both cases, Lemma \ref{strong} (iv) allows to conclude that $\beta=\theta(\alpha)$.
 
 Finally assume that $a<b$. Recall Eq. \ref{decomp2Oplus} and let $\alpha\in O^+\cap SQ$ and $\beta\in S_{\alpha}\cap O^+$. Then $\beta\in B_1^+\sqcup B_3^+\sqcup B_{2,\,r}^+$. If $\beta\in B_1^+\sqcup B_3^+$ we may conclude as above, by Eq. \ref{decompSQ}. Moreover it is not possible that $\beta\in B_{2,\,r}^+$ since otherwise $\alpha+\beta\in S\cap SQ$, which would imply that there exists some root in $S\cap SQ$ on the same line or column as $\beta$, which is not possible.
 
 Now if $\alpha\in O^+\cap(B_1^+\sqcup B_3^+)$ and $\beta\in S_{\alpha}\cap O^+$, we conclude as in case $a\ge b$.
 Suppose that $\alpha\in O^+\cap B_{2,\,r}^+$ and let $\beta\in S_{\alpha}\cap O^+$. Then $\alpha\in S_{\beta}\cap O^+$ and by the above one has necessarily that $\beta\in B_{2,\,r}^+\cap O^+$ too. Lemma \ref{strong} (iv) allows to conclude that $\beta=\theta(\alpha)$.
 
 Finally let $\alpha\in O^-$ and $\beta\in S_{\alpha}\cap O^-$. Recall Eq. \ref{decomp2Omoins}. By Lemma \ref{strong} (iv) we can conclude that $\beta=\theta(\alpha)$.
  \end{proof}
  
  \subsection{}\label{res}
  By the previous subsections (Lemmas \ref{basis}, \ref{lmstatroot} and \ref{lmtheta}), all conditions of Proposition \ref{propnondeg} are satisfied. Then the restriction to $\g_O\times\g_O$ of the bilinear form $\widetilde\Phi_y$ is nondegenerate. It follows by subsection \ref{T} that conditions (i), (ii), (iii) and (iv) of Lemma \ref{lmAP} are also satisfied. We have then obtained the following Lemma.
  \begin{lm}\label{herenondeg}
  With our notation, we have that:
  \begin{equation*}{\rm ad}^*\widetilde{\p_{\Lambda}}(y)\oplus\g_{-T}=\widetilde{\p_{\Lambda}}^*\end{equation*}
  and then, since $\dim\g_{-T}={\rm index}\,\widetilde{\p_{\Lambda}}$, $y$ is a regular element in $\widetilde{\p_{\Lambda}}^*$. Moreover there exists a unique element $h\in\h_{\pi'}\oplus\h_{\Lambda}$ such that, for all $\gamma\in S$, we have $\gamma(h)=1$. Hence $(h,\,y)$ is an adapted pair for $\widetilde{\p_{\Lambda}}$.
  \end{lm}
  
 To check conditions (v) and (vi) of Lemma \ref{lmAP}, it remains to compute the $t(\gamma)$ for all $\gamma\in T$ and then the formal character of $Y\big(\widetilde{\p_{\Lambda}}\big)$.
  
  \subsection{The formal character of the algebra of symmetric invariants associated with the contraction of the canonical truncation.}\label{formch}

Recall subsection \ref{T} that $T=(\Delta^+\sqcup\Delta^-_{\pi'})\setminus\bigsqcup_{\gamma\in S}\Gamma_{\gamma}$ is equal to 
\begin{align*}T=\{\ep_s-\ep_{n-s+1},\;\ep_{n-s}-\ep_k;\;s+1\le k\le n,\,k\neq n-s\}.\end{align*} 

For every $\gamma\in T$, we have to compute $t(\gamma)\in\mathbb QS$ such that $\gamma+t(\gamma)$ vanishes on $\h_{\pi'}\oplus\h_{\Lambda}$. In other words, by what we said in subsection \ref{truncCartan}, we have to compute $t(\gamma)\in\mathbb QS$ such that $\gamma+t(\gamma)$ is proportional to $\varpi_s+\varpi_{n-s}$.

Since conditions (v) and (vi) of Lemma \ref{lmAP} will be satisfied, this will imply  that $Y\big(\widetilde{\p_{\Lambda}}\big)$ has $y+\g_{-T}$ as a Weierstrass section and then is a polynomial algebra over $\mathbb K$ on ${\rm index}\,\p_{\Lambda}$ algebraically independent homogeneous generators of weight $\gamma+t(\gamma)$, $\gamma\in T$. By \cite[Remark 3.7.3]{F01} we also can give the degree of each homogeneous generator of this polynomial algebra. Indeed write, for each $\gamma\in T$, $t(\gamma)=\sum_{\alpha\in S}m_{\alpha,\,\gamma}\alpha$ with $m_{\alpha,\,\gamma}\in\mathbb N$. Setting $\lvert t(\gamma)\rvert =\sum_{\alpha\in S} m_{\alpha,\,\gamma}$, then  $\partial_{\gamma}=1+\lvert t(\gamma)\rvert$ is the degree of each homogeneous generator of weight $\gamma+t(\gamma)$. \smallskip

Recall the set $S$ in subsection \ref{S}.\smallskip

 Let $\gamma=\ep_s-\ep_{n+1-s}=\beta_s$. By Eq. \ref{sumortho}, we have $t(\beta_s)=\sum_{i=1}^{s-1}\beta_i\in\mathbb NS$ is such that $\beta_s+t(\beta_s)=\varpi_s+\varpi_{n-s}$. Moreover
$\partial_{\gamma}=s$.\smallskip

 Assume firstly that $a\ge b$. \smallskip

Let $\gamma=\ep_{n-s}-\ep_k\in T$ with $s+1\le k\le n-s-1$. Then
$$t(\gamma)=\gamma_s+\delta_{n-s-k}+\sum_{i=1}^{n-s-k-1}(\delta_i+\gamma_{s-i})+\sum_{j=1}^{2s-n+k}\beta_j\in\mathbb NS$$
is such that $\gamma+t(\gamma)=\sum_{j=1}^s\ep_j-\sum_{j=n+1-s}^n\ep_j=\varpi_s+\varpi_{n-s}$. Moreover $\partial_{\gamma}=n-k+1$.\smallskip

Let $\gamma=\ep_{n-s}-\ep_{n-s+k}\in T$ with $1\le k\le n-2s$. Then
$$t(\gamma)=\gamma_s+\sum_{i=1}^{s-k}\beta_i+\sum_{i=1}^{k-1}(\delta_i+\gamma_{s-i})\in\mathbb NS$$
is such that $\gamma+t(\gamma)=\varpi_s+\varpi_{n-s}$. Moreover $\partial_{\gamma}=s+k$.\smallskip

Let $\gamma=\ep_{n-s}-\ep_{n-s+k}\in T$ with $n-2s+1\le k\le s$. Write $k=s-u$ with $0\le 2u<3s-n$. Then
$$t(\gamma)=\gamma_s+\sum_{i=1}^{n-2s-1}(\delta_i+\gamma_{s-i})+\sum_{i=1}^{3s-n-u}\beta_i+\sum_{i=1}^u(-\beta''_i)+\sum_{i=1}^{u+1}(-\beta'_i)+\sum_{i=1}^u\beta_i\in\mathbb NS$$
is such that $\gamma+t(\gamma)=\varpi_s+\varpi_{n-s}.$ Moreover $\partial_{\gamma}=n+s-2k+1$.\smallskip

Let $\gamma=\ep_{n-s}-\ep_{n-s+k}\in T$ with $k=s-u$, $2u\ge 3s-n$ and $u\le 3s-n-1$. Then
\begin{align*}t(\gamma)=\gamma_s+\sum_{i=1}^{n-2s-1}(\delta_i+\gamma_{s-i})+\sum_{i=1}^u\beta_i+\sum_{i=1}^{3s-n-u}(-\beta''_i)+\sum_{i=1}^{3s-n-u}(-\beta'_i)\\
+\sum_{i=1}^{3s-n-u}\beta_i\in\mathbb NS\end{align*}
 is such  that $\gamma+t(\gamma)=\varpi_s+\varpi_{n-s}.$ Moreover $\partial_{\gamma}=3s-n+2k$.\smallskip
  
  Finally assume that $a<b$.\smallskip
  
  Let $\gamma=\ep_{n-s}-\ep_{s+k}\in T$ with $2\le 2k\le n-3s$. Then
  $$t(\gamma)=\sum_{i=0}^{s-1}\gamma_{s-i}+\sum_{i=1}^s\delta_i+\sum_{i=1}^k\beta'_i+\sum_{i=1}^{k-1}(-\beta''_i)\in\mathbb NS$$ is such that
  $\gamma+t(\gamma)=\varpi_s+\varpi_{n-s}.$ Moreover $\partial_{\gamma}=2s+2k$.\smallskip
  
  Let $\gamma=\ep_{n-s}-\ep_{s+k}\in T$ with $1\le k\le n-3s$ and $2k>n-3s$. Then
  $$t(\gamma)=\sum_{i=0}^{s-1}\gamma_{s-i}+\sum_{i=1}^s\delta_i+\sum_{i=1}^{n-3s-k}\beta'_i+\sum_{i=1}^{n-3s-k}(-\beta''_i)\in\mathbb NS$$ is such that
  $\gamma+t(\gamma)=\varpi_s+\varpi_{n-s}.$ Moreover $\partial_{\gamma}=1+2n-4s-2k$.\smallskip
  
  Let $\gamma=\ep_{n-s}-\ep_{s+k}\in T$ with $n-3s<k\le n-2s-1$. Then
  $$t(\gamma)=\gamma_s+\sum_{i=1}^{n-2s-k}\delta_i+\sum_{i=1}^{n-2s-k-1}\gamma_{s-i}+\sum_{i=1}^{n-2s-k}\beta_i\in\mathbb NS$$ is such that
  $\gamma+t(\gamma)=\varpi_s+\varpi_{n-s}.$ Moreover $\partial_{\gamma}=1+3n-6s-3k$.\smallskip
  
  Let $\gamma=\ep_{n-s}-\ep_{s+k}\in T$ with $n-2s<k\le n-s$. Then
  $$t(\gamma)=\gamma_s+\sum_{i=1}^{k-(n-2s)-1}(\delta_i+\gamma_{s-i})+\sum_{i=1}^{n-s-k}\beta_i\in\mathbb NS$$  is such that
  $\gamma+t(\gamma)=\varpi_s+\varpi_{n-s}.$ Moreover $\partial_{\gamma}=k-n+3s$.

  We have then obtained that condition (v) of Lemma \ref{lmAP} is satisfied. By Eq. \ref{upperbound} we have then, in both cases ($a\ge b$ or $a<b$) 
  \begin{lm}\label{caracbornesup}
  $${\rm ch}\,Y\big(\widetilde{\p_{\Lambda}}\big)\le (1-e^{\varpi_s+\varpi_{n-s}})^{-(n-s)}.$$
  \end{lm}
  It remains to compare this upper bound with the lower bound constructed in \cite{F00} to obtain condition (vi) of Lemma \ref{lmAP}.
  
  \subsection{The lower bound.}\label{lowerbound}
  
  In \cite[Thm. 9.9.2]{F00} we have obtained an inclusion of an $\h$-module and of an algebra, say $\mathscr B$, into the algebra generated by symmetric semi-invariants $Sy\big(\widetilde\p\big)$. 
  In other words we have that
  $$\mathscr B\subset Sy\big(\widetilde\p\big).$$
  Moreover we have $$Sy\big(\widetilde\p\big)=S\big(\widetilde\p\big)^{\widetilde\p'}$$ by \cite[Remark 2.3.1]{F01}
  and  we have 
  $$Sy\big(\widetilde\p\big)=Y\big(\widetilde\p_{\Lambda}\big)=S\big(\widetilde\p_{\Lambda}\big)^{\widetilde\p'}$$ by Eq. \ref{semitrunc} and since $\widetilde\p'\subset\widetilde\p_{\Lambda}$.
  
  Finally we have that
  $$\widetilde\p_{\Lambda}\subset\widetilde{\p_{\Lambda}}$$ by \cite[Lem. 2.6.4]{F01}. Then by the above, we have that
  $$Sy\big(\widetilde\p\big)\subset S\big(\widetilde{\p_{\Lambda}}\big)^{\widetilde\p'}.$$
  
  Moreover
  the weight vectors in $\mathscr B$ have their weight belonging to $\mathbb K\Lambda(\p)$ by \cite[Proof of Prop. 8.2.2]{F00} and \cite[7.1]{FJ2}. Hence their weight vanishes on $\h_{\Lambda}$ by subsection \ref{truncCartan}. It follows that
  $$\mathscr B\subset S\big(\widetilde{\p_{\Lambda}}\big)^{\widetilde\p'\oplus\h_{\Lambda}}=Y\big(\widetilde{\p_{\Lambda}}\big).$$
  
  As in \cite[Prop. 9.10.1]{F00} we deduce that
  $$\prod_{O_{\alpha}\in E(\pi')}\big(1-e^{\delta_{O_{\alpha}}}\big)^{-1}\le{\rm ch}\,Y\big(\widetilde{\p_{\Lambda}}\big)$$
  where $\delta_{O_{\alpha}}$ is given by Eq. \ref{poidsclas}.
  
  By Eq. \ref{weightclassique} and comparing with Lemma \ref{caracbornesup}
  we obtain the following.
  
  \begin{lm}\label{caracequal}
  
  We have that 
  $${\rm ch}\,Y\big(\widetilde{\p_{\Lambda}}\big)=(1-e^{\varpi_s+\varpi_{n-s}})^{-(n-s)}.$$
  \end{lm}
  
  Hence condition (vi) and finally all conditions of Lemma \ref{lmAP} are satisfied.  We can then conclude the following, using also Lemma \ref{casdiff} in case $3s\neq n$.
  
  \begin{thm}\label{thmWS}
  Let $n$ and $s$ be positive integers such that $2s<n$.
  Let $\p$ be the standard parabolic subalgebra of $\mathfrak s\mathfrak l_n(\mathbb K)$ whose standard Levi factor $\mathfrak r$ has two extremal blocks of size $s$ and one central block of size $n-2s$. Let $\widetilde{\p_{\Lambda}}=(\mathfrak r'\oplus\h_{\Lambda})\ltimes\m^a$ be the In\"on\"u-Wigner contraction of the canonical truncation $\p_{\Lambda}$ of $\p$ according to the decomposition $\p_{\Lambda}=(\mathfrak r'\oplus\h_{\Lambda})\oplus\m$ where $\mathfrak r'$ is the derived subalgebra of $\mathfrak r$, $\m$ is the nilpotent radical of $\p$ and $\h_{\Lambda}=\p_{\Lambda}\cap\h^{\pi\setminus\pi'}$, where $\h^{\pi\setminus\pi'}$ is the centre of $\mathfrak r$. Let $\widetilde\p=\mathfrak r\ltimes\m^a$ be the In\"on\"u-Wigner contraction of $\p$ with respect to the decomposition $\p=\mathfrak r\oplus\m$.
  \begin{enumerate}
\item[(i)]  The algebra $Y\big(\widetilde{\p_{\Lambda}}\big)$ of symmetric invariants in $S\big(\widetilde{\p_{\Lambda}}\big)$ under adjoint action has a Weierstrass section and hence is a polynomial $\mathbb K$-algebra. Its number of algebraically independent homogeneous generators is equal to $n-s$. Each of them has a weight equal to $\varpi_s+\varpi_{n-s}$ and a degree computed in subsection \ref{formch}. 
  
\item[(ii)]  In particular, when the two extremal blocks are not of the same size as the central block, the algebra $Sy\big(\widetilde\p\big)$ generated by symmetric semi-invariants in $S\big(\widetilde\p\big)$ under adjoint action is equal to the algebra $Y\big(\widetilde{\p_{\Lambda}}\big)$ and hence has a Weierstrass section and is a polynomial $\mathbb K$-algebra in $n-s$ generators.
\end{enumerate}
  \end{thm}
  
  \subsection{The sum of degrees.}\label{sumdeg}
  
  Let $c\big(\widetilde{\p_{\Lambda}}\big)=\frac{1}{2}(\dim\widetilde{\p_{\Lambda}}+{\rm index}\,\widetilde{\p_{\Lambda}})$. It is always an integer and one can check that
  $$c\big(\widetilde{\p_{\Lambda}}\big)=\sum_{\gamma\in T}\partial_{\gamma}=\frac{1}{2}(n-s)(n-s+1)+s^2-1.$$
  
  By \cite[thm. 5.7]{JS} the equality $c\big(\widetilde{\p_{\Lambda}}\big)=\sum_{\gamma\in T}\partial_{\gamma}$  was expected at least for the case $3s\neq n$. Indeed in this case $\widetilde{\p_{\Lambda}}=\widetilde\p_{\Lambda}$ by Lemma \ref{casdiff} and $\widetilde\p_{\Lambda}$ has an adapted pair by Lemma \ref{herenondeg}. Moreover $\sum_{\gamma\in T}\partial_{\gamma}$ is the sum of the degrees of all homogeneous generators of $Y\big(\widetilde{\p_{\Lambda}}\big)=Y\big(\widetilde\p_{\Lambda}\big)=Sy\big(\widetilde\p_{\Lambda}\big)=Sy\big(\widetilde\p\big)$ by Eq. \ref{semitrunc}.
  
  \subsection{The case when the three blocks are of the same size.}
  
When the three blocks of the standard Levi factor are of the same size, we have the following.  
  \begin{lm}\label{samesize}
  
  When $3s=n$ we have that $\widetilde\p_{\Lambda}=\widetilde\p'$ and then
  $$Sy\big(\widetilde\p\big)=Y\big(\widetilde\p'\big).$$
  
  \end{lm}
  
  \begin{proof}
  
  Recall Eq. \ref{semitrunc} that $Sy\big(\widetilde\p\big)=Y\big(\widetilde\p_{\Lambda}\big)$ where $\widetilde\p_{\Lambda}$ is the canonical truncation of $\widetilde\p$.
  Moreover one always has that $\widetilde\p'\subset\widetilde\p_{\Lambda}$. Then $\widetilde\p_{\Lambda}/\widetilde\p'\subset\h^{\pi\setminus\pi'}$.
  
  When $3s=n$, the three blocks of the standard Levi factor of $\p$ are of the same size $a=s$ and $\pi'=\pi\setminus\{\alpha_a,\,\alpha_{2a}\}$. Moreover one can check that the three determinants in $S(\m)$ which lie in $R_1$, $SQ$ or $R_2$ with notation of subsection 
  \ref{ST} are semi-invariant, actually they belong to $S(\m)^{\mathfrak r'}\subset Sy\big(\widetilde\p\big)$.
  
  The weight (viewed as an element of $(\h^{\pi\setminus\pi'})^*$) of each of these determinants is respectively equal to $\lambda_1=2\varpi_a-\varpi_{2a}$, $\lambda_2=\varpi_a+\varpi_{2a}$ and $\lambda_3=2\varpi_{2a}-\varpi_a$.
  
  Then by Eq. \ref{deftrunc} one has that $\widetilde\p_{\Lambda}/\widetilde\p'\subset\cap_{1\le i\le 3}\ker(\lambda_i)=\{0\}$. Hence $\widetilde\p_{\Lambda}=\widetilde\p'$.
    \end{proof}
  \begin{Rq}\label{polynomialitysamesize}
  
  We cannot say for the moment whether $Sy\big(\widetilde\p\big)=Y\big(\widetilde\p'\big)$ is or not polynomial, in case $3s=n$. 
  Indeed we have that
  $$Y\big(\widetilde{\p_{\Lambda}}\big)=S\big(\widetilde{\p_{\Lambda}}\big)^{\widetilde{\p_{\Lambda}}}\subset S\big(\widetilde\p\big)^{\widetilde\p'}=Sy\big(\widetilde\p\big).$$
  
  Hence $${\rm GKdim}\,Y\big(\widetilde{\p_{\Lambda}}\big)\le{\rm GKdim}\,Sy\big(\widetilde\p\big)={\rm GKdim}\,Y\big(\widetilde\p'\big).$$
  
 By Thm. \ref{thmWS} ${\rm GKdim}\,Y\big(\widetilde{\p_{\Lambda}}\big)=n-s=2s$, whilst ${\rm GKdim}\,Y\big(\widetilde\p'\big)={\rm index}\,\widetilde\p'=2s+1$ by Eq. \ref{GKdimindex}, since it is easy to check that $\widetilde\p'_{\Lambda}=\widetilde\p'$ and by Lemma \ref{indexindex1}.  It follows that the formal character of $Y\big(\widetilde{\p_{\Lambda}}\big)$ computed in Lemma \ref{caracequal} is not sufficient to conclude for polynomiality of $Sy\big(\widetilde\p\big)$ in the present case.

 In case $s=1$ and $n=3$, $\p$ is the standard Borel subalgebra of $\mathfrak s\mathfrak l_3$ and then $\m=\n^+_{\pi}$ with notation in Section \ref{Not}. We obtain in this case that $\widetilde\p'=\m$ and then $Sy\big(\widetilde\p\big)=Y\big(\widetilde\p'\big)=S(\m)$  is polynomial (see also subsection \ref{caseBorel}). 
  \end{Rq}

\end{document}